\documentclass[a4paper,10pt]{amsart}

\usepackage[T1]{fontenc}
\usepackage[bookmarks]{hyperref}
\usepackage{wasysym} 
\usepackage{amsfonts}
\usepackage{amsmath}
\usepackage{amssymb,amsxtra,amscd}
\usepackage{mathrsfs}
\usepackage{mathtools}
\usepackage{todonotes}
\usepackage{lmodern}
\usepackage[foot]{amsaddr}
\usepackage{scalerel,stackengine}
\stackMath
\newcommand\reallywidehat[1]{%
	\savestack{\tmpbox}{\stretchto{%
			\scaleto{%
				\scalerel*[\widthof{\ensuremath{#1}}]{\kern-.6pt\bigwedge\kern-.6pt}%
				{\rule[-\textheight/2]{1ex}{\textheight}}
			}{\textheight}%
		}{0.5ex}}%
	\stackon[1pt]{#1}{\tmpbox}%
}

\newcommand{\R}{\mathbb{R}}
\newcommand{\C}{\mathbb{C}}

\newcommand{\N}{\mathbb{N}}

\newcommand{\K}{\mathbb{K}}

\newcommand{\supp}{\mbox{supp}\,}
\newcommand{\dist}{\mbox{dist}\,}

\newtheorem{theorem}{Theorem}[section]

\newtheorem{corollary}[theorem]{Corollary}
\newtheorem{proposition}[theorem]{Proposition}

\theoremstyle{definition}
\newtheorem{remark}[theorem]{Remark}
\newtheorem{definition}[theorem]{Definition}
\newtheorem{example}[theorem]{Example}

\title[Power bounded weighted composition operators]{Power bounded weighted composition operators on function spaces defined by local properties}

\author[T.\ Kalmes]{T.\ Kalmes}
\address{Technische Universit\"at Chemnitz, Fakult\"at f\"ur Mathematik, 09107 Chemnitz, Germany}
\email{thomas.kalmes@mathematik.tu-chemnitz.de}

\begin{document}

\begin{abstract}
	We study power boundedness and related properties such as mean ergodicity for (weighted) composition operators on function spaces defined by local properties. As a main application of our general approach we characterize when (weighted) composition operators are power bounded, topologizable, and (uniformly) mean ergodic on kernels of certain linear partial differential operators including elliptic operators as well as non-degenerate parabolic operators. Moreover, under mild assumptions on the weight and the symbol we give a characterization of those weighted composition operators on the Fr\'echet space of continuous functions on a locally compact, $\sigma$-compact, non-compact Hausdorff space which are generators of strongly continuous semigroups on these spaces.\\
	
	
\end{abstract}

\maketitle

\section{Introduction}

The aim of this article is to present a general approach to power boundedness and topologizability of weighted composition operators $C_{w,\psi}(f)=w\cdot(f\circ \psi)$ acting on locally convex spaces of scalar valued functions $f$ which are defined by local properties. 

Recall that a continuous linear operator $T$ on a Hausdorff locally convex space $E$ is power bounded precisely when the set of all its iterates is equicontinuous. This notion is closely connected with $T$ being mean ergodic, i.e.\ with the property that for every $x\in E$ the sequence $(T_{[n]}(x))_{n\in\N}$ of Ces\`aro means
$$T_{[n]}(x)=\frac{1}{n}\sum_{m=1}^n T^m(x)$$
converges; in the Banach space setting, see for example \cite{Krengel}, \cite{FonfLinWojtaszczyk01}, and references therein, for the setting of more general locally convex spaces, see \cite{Yosida}, \cite{AlbaneseBonetRicker09}, \cite{AlbaneseBonetRicker09b}, and references therein.

Likewise to power boundedness of an operator being a sufficient condition for mean ergodicity in many situations, our interest in topologizable operators, or more precisely, $m$-topologizable operators (for the precise definitions, see section \ref{power boundedness} below) comes from the fact that due to a recent result by Goli\'nska and Wegner \cite{GolinskaWegner}, $m$-topologizability of an operator implies that it generates a uniformly continuous operator semigroup. It should be noted that although the definitions and most basic results for semigroups of operators on locally convex spaces are the same as for Banach spaces (see e.g.\ \cite{Yosida}) not every continuous linear operator on a locally convex space generates a strongly continuous semigroup, see \cite{FrerickJordaKalmesWengenroth}.

Recently, mean ergodicity, power boundedness, and topologizablity of (weighted) composition operators and multiplication operators on various function spaces have been investigated by several authors, see e.g.\ \cite{BeltranGomezJordaJornet16b}, \cite{BeltranGomezJordaJornet16}, \cite{BonetDomanski11}, \cite{BonetDomanski11b}, \cite{BonetJordaRodriguez18}, \cite{BonetRicker09}, \cite{GomezJordaJornet16} \cite{Wolf12a}, \cite{Wolf12b}, \cite{Wolf15}.

The purpose of this article is to give a general framework for studying power boundedness etc.\ of weighted composition operators. In section \ref{local function spaces} we recall the notion of general locally convex sheaves of functions which gives the appropriate general framework for our objective. This general approach enables us to give a necessary condition of topologizability in section \ref{power boundedness} which is also sufficient provided that the space of functions under consideration carries the compact-open topology (see Corollary \ref{characterizing topologizability for compact-open}) or the standard $C^r$-topology in case of spaces of continuously differentiable functions (see Corollary \ref{characterizing topologizabilty for smooth}). These results also permit to characterize power boundedness in these cases in terms of the weight $w$ and symbol $\psi$ of the operator (see Theorems \ref{power boundedness for compact-open} and \ref{power boundedness for smooth}, respectively) as well as (uniform) mean ergodicity for unweighted composition operators on a large class of function spaces which are Fr\'echet-Montel spaces when equipped with the compact-open topology (see Corollary \ref{mean ergodic for compact-open}).

As a main application of our general results, in section \ref{kernels of differential operators} we characterize power boundedness of weighted composition operators on kernels of certain (hypoelliptic) partial differential operators in $C^\infty(X)$ for open subsets $X$ of $\R^d$ in terms of the weight and the symbol of the operator (see Theorems \ref{power bounded elliptic} and \ref{one-dimensional characteristics}, respectively). Moreover, for (unweighted) composition operators we show that the notions of topologizability, power boundedness and (uniform) mean ergodicity coincide on kernels of these partial differential operators. In particular, the results of this section generalize results obtained in \cite{BeltranGomezJordaJornet16} and \cite{BonetDomanski11}, where the special case of the Cauchy-Riemann operator was considered.

In the final section \ref{generators} we characterize, under mild additional assumptions on the weight and the symbol, those weighted composition operators which are generators of strongly continuous operator semigroups on the space of continuous functions $C(X)$ equipped with the compact-open topology, where $X$ is a locally compact, $\sigma$-compact, non-compact Hausdorff space.\\

Results about dynamical properties of weighted composition operators like transitivity/hypercyclicity and (weak) mixing in the same general setup of function spaces defined by local properties and applications to concrete function spaces have been investigated in \cite{Kalmes17}.\\ 

Throughout the paper, we use standard notation and terminology from functional analysis. For anything related to functional analysis which is not explained in the text we refer the reader to \cite{MeiseVogt}. Moreover, we use common notation from the theory of linear partial differential operators. For this we refer the reader to \cite{Hoermander}.

By an open, relatively compact exhaustion $(X_n)_{n\in\N}$ of a topological space $X$ we understand a sequence of open subsets of $X$ such that $\overline{X}_n\subseteq X_{n+1}$ with compact closure $\overline{X}_n$ for all $n\in\N$ such that $\cup_{n\in\N} X_n=X$.

\section{Function spaces defined by local properties}\label{local function spaces}

In order to deal with weighted composition operators on several function spaces at once we use the notion of sheaves. Here and in the sequel let $\K\in\{\R,\C\}$.

\begin{definition}\label{sheaf}
\begin{rm}
	For a locally compact, $\sigma$-compact, non-compact Hausdorff space $\Omega$ let $\mathscr{F}$ be a \textit{sheaf of functions} on $\Omega$, i.e.\
	\begin{itemize}
		\item For every open subset $X\subseteq\Omega$ we have a vector space $\mathscr{F}(X)$ of $\K$-valued functions such that whenever $Y\subseteq\Omega$ is another open set with $Y\subseteq X$ the restriction mapping
		$$r_X^Y:\mathscr{F}(X)\rightarrow\mathscr{F}(Y),f\mapsto f_{|Y}$$
		is well-defined.
		
		\item (Localization) For an open set $X\subseteq\Omega$, for every open cover $(X_\iota)_{\iota\in I}$ of $X$, and for each $f,g\in\mathscr{F}(X)$ with $f_{|X_\iota}=g_{|X_\iota} (\iota\in I)$ we have $f=g$.
		
		\item (Gluing) For an open set $X\subseteq\Omega$, for every open cover $(X_\iota)_{\iota\in I}$ of $X$, and for all $(f_\iota)_{\iota\in I}\in\prod_{\iota\in I}\mathscr{F}(X_\iota)$ with $f_{\iota|X_\iota\cap X_\kappa}=f_{\kappa|X_\iota\cap X_\kappa}\, (\iota,\kappa\in I)$ there is $f\in\mathscr{F}(X)$ with $f_{|X_\iota}=f_\iota\,(\iota\in I)$.
	\end{itemize}
	It follows immediately from the above properties that for every open subset $X\subseteq\Omega$ and each open, relatively compact exhaustion $(X_n)_{n\in\N_0}$ of $X$ the spaces $\mathscr{F}(X)$ and the projective limit $\mbox{proj}_{\leftarrow n}(\mathscr{F}(X_n),r_{X_{n+1}}^{X_n})$, i.e.\ the subspace
	$$\{(f_n)_{n\in\N_0}\in\prod_{n\in\N_0}\mathscr{F}(X_n);\,\forall\,n\in\N_0: f_n=r_{X_{n+1}}^{X_n}(f_{n+1})\}$$
	of $\prod_{n\in\N_0}\mathscr{F}(X_n)$ are algebraically isomorphic via the mapping
	$$\mathscr{F}(X)\rightarrow\mbox{proj}_{\leftarrow n}(\mathscr{F}(X_n),r_{X_{n+1}}^{X_n}),f\mapsto (r_X^{X_n}(f))_{n\in\N_0}=(f_{|X_n})_{n\in\N_0}.$$
	Indeed, injectivity follows from the localization property and surjectivity from the gluing property of a sheaf.
	
	For obvious reasons, $\mathscr{F}(X)$, $X\subseteq\Omega$ open, is what we call a \textit{space of functions defined by local properties}. Since we want to have some results from functional analysis at our disposal, we define the following property for a sheaf of functions $\mathscr{F}$ on $\Omega$:
	
	\begin{enumerate}
		\item[($\mathscr{F}1$)] The function space $\mathscr{F}(X)$, where $X\subseteq\Omega$ is open, is a webbed and ultrabornological Hausdorff locally convex space (which is satisfied, for example, if $\mathscr{F}(X)$ is a Fr\'echet space). Additionally, we assume that $\mathscr{F}(X)\subseteq C(X)$, the latter denoting the space of $\K$-valued continuous functions on $X$, and we assume that for each $x\in X$ the point evaluation $\delta_x$ in $x$ is a continuous linear functional on $\mathscr{F}(X)$.
		
		Whenever $X,Y\subseteq\Omega$ are open with $Y\subseteq X$ it therefore follows that the restriction map $r_X^Y$ has closed graph, hence is continuous by De Wilde's Closed Graph Theorem (see e.g.\ \cite[Theorem 24.31]{MeiseVogt}).
		
		Moreover, for every open, relatively compact exhaustion $(X_n)_{n\in\N_0}$ of $X$ we assume that the above mentioned algebraic isomorphism between $\mathscr{F}(X)$ and $\mbox{proj}_{\leftarrow n}(\mathscr{F}(X_n),r_{X_{n+1}}^{X_n})$ is even a topological isomorphism.
	\end{enumerate}
\end{rm}
\end{definition}

\begin{remark}\label{properties of local function sheaves}
\begin{rm}
	For a sheaf $\mathscr{F}$ on $\Omega$ with the property that $\mathscr{F}(X)$ is a locally convex space and $r_X^Y$ is continuous for each open $Y\subseteq X\subseteq\Omega$ (by definition, this means that $\mathscr{F}$ is a locally convex sheaf of functions on $\Omega$) it is immediate that for an arbitrary open, relatively compact exhaustion $(X_n)_{n\in\N_0}$ of $X$ the canonical isomorphism between $\mathscr{F}(X)$ and $\mbox{proj}_{\leftarrow n}(\mathscr{F}(X_n),r_{X_{n+1}}^{X_n})$ is continuous. Therefore, if $\mathscr{F}$ is a locally convex sheaf of continuous functions such that $\mathscr{F}(X)$ is a Fr\'echet space for each open $X\subseteq\Omega$ on which $\delta_x$ is a continuous linear functional for every $x\in X$, it follows from the Open Mapping Theorem and the fact that Fr\'echet spaces are ultrabornological (see e.g.\ \cite[Remark 24.15 c)]{MeiseVogt}) and webbed (see e.g.\ \cite[Corollary 24.29]{MeiseVogt}) that $(\mathscr{F}1)$ is satisfied. 
	
	Moreover, for a sheaf $\mathscr{F}$ on $\Omega$ satisfying $(\mathscr{F}1)$ it follows from $\delta_x\in\mathscr{F}(X)'$ for each $x\in X$ that the inclusion mapping
	$$\mathscr{F}(X)\hookrightarrow C(X),f\mapsto f$$
	has closed graph - where we equip $C(X)$ as usual with the compact-open topology. Since $\mathscr{F}(X)$ is supposed to be ultrabornological it follows from De Wilde's Closed Graph Theorem that this inclusion is continuous, i.e.\ the topology carried by $\mathscr{F}(X)$ is finer than the compact-open topology.
\end{rm}
\end{remark}

\begin{example}\label{examples of sheaves}
\hspace{2em}
\begin{rm}
	\begin{itemize}
		\item[i)] For every $\sigma$-compact, locally compact, non-compact Hausdorff space $\Omega$ the sheaf $C$ of continuous functions satisfies $(\mathscr{F}1)$, provided that $C(X)$ is as usual equipped with the compact-open topology, that is, the locally convex topology defined by the family of seminorms $\{\|\cdot\|_K;\,K\subseteq X\mbox{ compact}\}$, where for a compact subset $K\subseteq X$
		$$\forall\,f\in C(X):\,\|f\|_K:=\sup_{x\in K}|f(x)|.$$
		
		\item[ii)] For $\Omega=\R^d$ and $r\in\N_0\cup\{\infty\}$ we denote by $C^r$ the sheaf of $r$-times continuously differentiable $\K$-valued functions. As usual, we equip $C^r(X)$ with the topology of local uniform convergence of all partial derivatives up to order $r$, i.e.\ the locally convex topology defined by the family of seminorms $\{\|\cdot\|_{l,K};\,l\in\N_0, l<r+1, K\subseteq X\mbox{ compact}\}$, where for $l<r+1$ and $K\subseteq X$ compact
		$$\forall\,f\in C^r(X):\|f\|_{l,K}:=\sup_{|\alpha|\leq l}\sup_{x\in K}|\partial^\alpha f(x)|,$$
		with $|\alpha|=\alpha_1+\ldots+\alpha_d$ denoting the length of a multi-index $\alpha\in\N_0^d$. In this way, we obtain a separable Fr\'echet space and the sheaf $C^r$ on $\R^d$ is easily seen to satisfy $(\mathscr{F}1)$.
		
		\item[iii)] For $\Omega=\C^d$ let $\mathscr{H}$ be the sheaf of holomorphic functions, i.e.\ for $X\subseteq\C^d$ let $\mathscr{H}(X)$ denote the space of holomorphic functions on $X$ endowed with the compact-open topology. Then $\mathscr{H}(X)$ is a separable nuclear Fr\'echet space and it follows easily that $(\mathscr{F}1)$ is satisfied.
		
		\item[iv)] For $\Omega=\R^d$ and a complex coefficient polynomial $P$ in $d$ variables, i.e.\ $P\in\C[X_1,\ldots,X_d]$, we define for an open $X\subseteq\R^d$
		$$C^\infty_P(X):=\{f\in C^\infty(X);\, P(\partial)f=0 \mbox{ in }X\},$$
		where as usual for $P(\xi)=\sum_{|\alpha|\leq m}a_\alpha\xi^\alpha$ we set
		$$\forall\,f\in C^\infty(X), x\in X:\,P(\partial)f(x)=\sum_{|\alpha|\leq m}a_\alpha\partial^\alpha f(x).$$
		Obviously, $C_P^\infty(X)$ is a subspace of $C^\infty(X)$ and since $P(\partial)$ is a continuous linear operator on the separable nuclear Fr\'echet space $C^\infty(X)$ it follows that $C_P^\infty(X)$ is a separable nuclear Fr\'echet space when equipped with the relative topology of $C^\infty(X)$. It is easily seen that $C_P^\infty$ is a sheaf on $\R^d$ which satisfies $(\mathscr{F}1)$. 
		
		Considering the special cases of $P(\partial)$ being the Cauchy-Riemann operator or the Laplace operator, we obtain $C_P^\infty$ is the sheaf of holomorphic functions $\mathscr{H}$ on open subsets of $\C$ or the sheaf of harmonic functions on open subsets of $\R^d$, respectively. Since these differential operators are elliptic it follows that the compact-open topology and the subspace topology inherited from $C^\infty(X)$ coincide on $C_P^\infty(X)$ so that iii) above in case of $d=1$ is indeed a special case of this example of sheaf, see section \ref{kernels of differential operators} below. 
	\end{itemize}
\end{rm}
\end{example}

\textbf{General assumption.}
Let $\mathscr{F}$ be a sheaf on $\Omega$ satisfying $(\mathscr{F}1)$, $X\subseteq \Omega$ open, and let $w:X\rightarrow\K$ as well as $\psi:X\rightarrow X$ be continuous. We assume that the weighted composition operator
$$C_{w,\psi}:\mathscr{F}(X)\rightarrow\mathscr{F}(X), f\mapsto w\cdot(f\circ\psi)$$
is well-defined and we call $w$ the \textit{weight} and $\psi$ the \textit{symbol} of $C_{w,\psi}$. In case of $w=1$ we simply denote the (unweighted) composition operator by $C_{\psi}$ instead of $C_{1, \psi}.$

For every $x\in X$ we have by hypothesis that $\delta_x\in\mathscr{F}(X)'$ and it follows easily from the Hahn-Banach Theorem, that the linear span of $\{\delta_x;\,x\in X\}$ is weak*-dense in $\mathscr{F}(X)'$. Since $\mathscr{F}(X)$ is Hausdorff this yields that $C_{w,\psi}$ has closed graph. By $(\mathscr{F}1)$ we conclude from De Wilde's Closed Graph Theorem \cite[Theorem 24.31]{MeiseVogt} that $C_{w,\psi}$ is continuous. For brevity, we write for a general $f\in C(X)$ also $C_{w,\psi}(f)$ instead of $w\cdot(f\circ\psi)$.

\section{Power boundedness and related properties for weighted composition operators}\label{power boundedness}

In this section we give some results concerning power boundedness and related properties for a weighted composition operator defined on a space of functions defined by local properties.

\begin{definition}
\begin{rm}
Let $E$ be a locally convex space and $T$ a continuous linear operator on $E$. We denote the set of all continuous seminorms on $E$ by $cs(E)$.
\begin{itemize}
	\item[i)] $T$ is called \textit{topologizable} if for every $p\in cs(E)$ there is $q\in cs(E)$ such that for every $m\in\N$ there is $\gamma_m>0$ with
	$$p(T^m(x))\leq \gamma_m q(x)\mbox{ for all }x\in E.$$
	\item[ii)] $T$ is called \textit{$m$-topologizable} if for every $p\in cs(E)$ there is $q\in cs(E)$ and $\gamma>0$ such that for every $m\in\N$
	$$p(T^m(x))\leq \gamma^m q(x)\mbox{ for all }x\in E.$$
	\item[iii)] $T$ is called \textit{power bounded} if the set of iterates of $T$, i.e.\  $\{T^m;\,m\in\N\}$, is equicontinuous. Thus $T$ is power bounded precisely when for every $p\in cs(E)$ there is $q\in cs(E)$ such that for every $m\in\N$
	$$p(T^m(x))\leq q(x)\mbox{ for all }x\in E.$$
	\item[iv)] We denote by
	$$T_{[n]}:=\frac{1}{n}\sum_{m=1}^n T^m,\,n\in\N,$$
	the \textit{Ces\`aro means} of $T$ and $T$ is called \textit{mean ergodic} if $(T_{[n]})_{n\in\N}$ converges in the strong operator topology, i.e.\ if for each $x\in E$ the sequence $(T_{[n]}(x))_{n\in\N}$ converges in $E$. The limit will be denoted by $P_T (x)$. In case $(T_{[n]})_{n\in\N}$ converges uniformly on bounded subsets of $E$ we call $T$ \textit{uniformly mean ergodic}.
\end{itemize}
\end{rm}
\end{definition}

\begin{remark}
	\hspace{2em}
	\begin{enumerate}
		\item[i)] Topologizable operators have been introduced by \.{Z}elazko in \cite{Zelazko07} (see also \cite{Bonet07}). It is known that if for a Hausdorff locally convex space $E$ the algebra $L(E)$ of all continuous endomorphisms of $E$ is topologizable, i.e.\ $L(E)$ admits a locally convex topology for which multiplication is jointly continuous, $E$ is necessarily subnormed, i.e.\ there is a norm on $E$ such that the corresponding topology is finer than the locally convex topology initially given on $E$, see \cite{Zelazko02} and references therein. In case of a sequentially complete $E$ it has been shown in \cite{Zelazko02} that this necessary condition on $E$ is also sufficient for the topologizability of $L(E)$. Motivated by this, in \cite{Zelazko07} it was investigated when for a given continuous linear operator $T$ on a locally convex Hausdorff space $E$ there is a unital subalgebra $A$ of $L(E)$ which contains $T$ and which admits a locally convex topology making $A$ into a topological algebra such that additionally the map
		$$A\times E\rightarrow E, (S,x)\mapsto Sx$$
		is continuous. By \cite[Theorem 5]{Zelazko07} for a given $T\in L(E)$ there is such a subalgebra $A$ of $L(E)$ if and only if $T$ is topologizable.
		
		\item[ii)] Clearly, every $m$-topologizable operator is topologizable and $T$ is topologizable whenever there is a sequence $(\alpha_m)_{m\in\N}$ of strictly positive numbers such that the set $\{\alpha_m T^m;\,m\in\N\}$ is equicontinuous and then the sequences $(\gamma_m)_{m\in\N}$ in the definition of topologizability can be chosen independently of the involved seminorms $p$ and $q$. Obviously, every power bounded operator is $m$-topologizable.
		
		\item[iii)] Since a set of operators on a barreled locally convex space is equicontinuous if and only if it is bounded in the strong operator topology, an operator $T$ on a barreled space $E$ is topologizable if there is a sequence $(\alpha_m)_{m\in\N}$ of strictly positive numbers such that $\{\alpha_m T^m(x);\,m\in\N\}$ is a bounded subset of $E$ for each $x\in E$. From the equality
		$$\frac{1}{n} T^n=T_{[n]}-\frac{n-1}{n}T_{[n-1]}$$
		it thus follows in particular that every mean ergodic operator on a barreled locally convex space is topologizable.
		
		\noindent If the sheaf $\mathscr{F}$ on $\Omega$ satisfies $(\mathscr{F}1)$ it follows that $\mathscr{F}(X)$ is ultrabornological, hence barreled, for every open $X\subseteq\Omega$.
		
		\item[iv)] $m$-topologizable operators have also been considered in \cite{Zelazko07}. However, our interest in $m$-topologizable operators is motivated by a recent result due to Goli\'nska and Wegner by which an $m$-topologizable operator $T$ on a sequentially complete locally convex space $E$ generates a uniformly continuous semigroup of operators on $E$. In fact, an even more general condition on $T$ suffices to be the generator of a uniformly continuous semigroup, see \cite[Theorem 1]{GolinskaWegner}. It should be noted that contrary to the Banach space setting not every continuous linear operator on a (sequentially complete) locally convex space generates a strongly continuous semigroup, see \cite{FrerickJordaKalmesWengenroth}.
		
		\item[v)] It is well known that a power bounded operator $T$ on a locally convex space $E$ is mean ergodic precisely when
		$$E=\mbox{ker}(I-T)\oplus\overline{\mbox{im}(I-T)},$$
		where $I$ denotes the identity on $E$, $\mbox{im}(I-T)$ the range of $I-T$, and $\oplus$ denotes an algebraic direct sum; this can easily be deduced from \cite[Chapter VIII, Sect.\ 3, Theorem 1]{Yosida}. For a power bounded mean ergodic operator $T$ the mapping $P_T$ then is the projection onto $\mbox{ker}(I-T)$ along $\overline{\mbox{im}(I-T)}$. While power bounded operators are not mean ergodic in general (see e.g.\ \cite[Example 2.3]{Yahdi}), on DF-spaces and LF-spaces which are also Montel spaces every power bounded operator is uniformly mean ergodic by \cite[Proposition 2.8]{AlbaneseBonetRicker09}.
	\end{enumerate}
\end{remark}

Before we state our first result which gives a necessary condition on the symbol $\psi$ for a weighted composition operator $C_{w,\psi}$ to be topologizable we recall:

\begin{definition}
	A continuous mapping $\psi:X\rightarrow X$ on a topological space $X$ is said to have \textit{stable orbits} if
	$$\forall\,K\subseteq X\mbox{ compact}\,\exists\,L\subseteq X\mbox{ compact}\,\forall\,m\in\N:\,\psi^m(K)\subseteq L,$$
	where $\psi^m$ denotes the $m$-fold iterate of $\psi$.
\end{definition}

\begin{proposition}\label{stable necessary}
	Let $\mathscr{F}$ satisfy $(\mathscr{F}1)$ and let $X\subseteq\Omega$ be open. Assume that additionally the following conditions hold.
	\begin{enumerate}
		\item[a)] There is an open, relatively compact exhaustion $(X_n)_{n\in\N}$ of $X$ such that for each $n\in\N$ and every $x\in X\backslash\overline{X}_n$ as well as every open $W\subseteq X\backslash\overline{X}_n$ containing $x$ there is $U\subseteq W$ open, $x\in U$ for which the restriction $r_X^{X_n\cup U}$ has dense range.
		\item[b)] $\mbox{ker }\delta_x\neq \mathscr{F}(X)$ for each $x\in X$.
		\item[c)] For every $m\in\N_0$ the set
		$$\{y\in X;\,w(\psi^m(y))\neq 0\}$$
		is dense in $X$.
	\end{enumerate}
	If $C_{w,\psi}$ is topologizable, then $\psi$ has stable orbits.
\end{proposition}

\begin{proof}
	Let $(X_n)_{n\in\N}$ be an open, relatively compact exhaustion of $X$ as in a) and set $X_0:=\emptyset$. Since $C_{w,\psi}$ is topologizable and because $\mathscr{F}$ satisfies $(\mathscr{F}1)$ it follows that for every $n\in\N$ and every zero neighborhood $U_n$ in $\mathscr{F}(X_n)$ there is $l\in\N$ and a zero neighborhood $U_l$ in $\mathscr{F}(U_l)$ (see \cite[Chapter 3.3]{Wengenroth}) such that for every $m\in\N$ there is $\gamma_m>0$ such that
	$$C_{w,\psi}^m\big((r_X^{X_l})^{-1}(U_l)\big)\subseteq\gamma_m(r_X^{X_n})^{-1}(U_n),$$
	which by taking polars implies
	$$\frac{1}{\gamma_m}\Big((r_X^{X_n})^{-1}(U_n))\Big)^\circ\subseteq \big((C_{w,\psi}^t)^m\big)^{-1}\Big(\big((r_X^{X_l})^{-1}(U_l)\big)^\circ\Big),$$
	i.e.\
	\begin{equation}\label{transpose inclusion}
		(C_{w,\psi}^t)^m\Big(\big((r_X^{X_n})^{-1}(U_n)\big)^\circ\Big)\subseteq\gamma_m\Big((r_X^{X_l})^{-1}(U_l)\Big)^\circ,
	\end{equation}
	where $C_{w,\psi}^t$ denotes the transpose of $C_{w,\psi}$.
	
	Because $\mathscr{F}$ satisfies $(\mathscr{F}1)$ the inclusion
	$$i_n:\mathscr{F}(X_n)\hookrightarrow C(X_n)$$
	is continuous so that with $\|f\|_{\overline{X}_{n-1}}:=\sup_{x\in\overline{X}_{n-1}}|f(x)|$ for $f\in C(X_n)$ the set
	$$U_n:=i_n^{-1}\big(\{f\in C(X_n);\,\|f\|_{\overline{X}_{n-1}}\leq 1\}\big)$$
	is a zero neighborhood in $\mathscr{F}(X_n)$ and
	$$\forall\,x\in X_{n-1}:\,\delta_x\in\Big((r_X^{X_n})^{-1}(U_n)\Big)^\circ.$$
	
	Let $l\in\N$, $U_l\in\mathscr{F}(X_l)$, and $(\gamma_m)_{m\in\N}$ be as above for this particular choice of $U_n$. Fix $m\in\N$. Since by hypothesis c)
	$$\bigcap_{j=0}^m\{y\in X;\,w(\psi^j(y))\neq 0\}$$
	is dense in $X$ there is $x_0\in X_{n-1}$ with $\prod_{j=0}^{m-1}|w(\psi^j(x_0))|\neq 0$. Assuming that $\psi^m(x_0)\notin \overline{X}_l$ we will derive a contradiction.
	
	Let $W\subseteq X\backslash\overline{X}_l$ be an open neighborhood of $\psi^m(x_0)$. By hypothesis b), there is in particular $g\in \mathscr{F}(W)$ such that $g(\psi^m(x_0))=1$. Moreover, applying hypothesis a), there is an open subset $U$ of $W$ containing $\psi^m(x_0)$ such that $r_X^{X_l\cup U}$ has dense range and such that $|g(y)|>1/2$ for every $y\in U$.
	
	By the properties of a sheaf, for every $M>0$ there is $h_M\in\mathscr{F}(X_l\cup U)$ such that
	$$r_{X_l\cup U}^{X_l}(h_M)=0\mbox{ and }r_{X_l\cup U}^{U}(h_M)=M g.$$
	Denoting the open ball around $M$ with radius $M/2$ in $\K$ by $B(M,M/2)$ it follows that
	$$\delta_{\psi^m(x_0)}^{-1}\big(B(M,\frac{M}{2})\big)\cap (r_{X_l\cup U}^{X_l})^{-1}(U_l)$$
	is a neighborhood of $h_M$ in $\mathscr{F}(X_l\cup U)$. It follows from the dense range of $r_X^{X_l\cup U}$ that there is $f_M\in\mathscr{F}(X)$ with
	$$r_X^{X_l\cup U}(f_M)-h_M\in \delta_{\psi^m(x_0)}^{-1}\big(B(M,\frac{M}{2})\big)\cap (r_{X_l\cup U}^{X_l})^{-1}(U_l).$$
	Hence,
	$$U_l\ni r_{X_l\cup U}^{X_l}\big(r_X^{X_l\cup U}(f_M)-h_M\big)=r_{X_l\cup U}^{X_l}\big(r_X^{X_l\cup U}(f_M)-h_M\big)+r_{X_l\cup U}^{X_l}(h_M)=r_X^{X_l}(f_M)$$
	so that $f_M\in (r_X^{X_l})^{-1}(U_l)$. Because $\delta_{x_0}\in\Big((r_X^{X_n})^{-1}(U_n)\Big)^\circ$ it follows from (\ref{transpose inclusion}) that for every $m\in\N$ and $M>0$
	\begin{align*}
		\gamma_m &\geq |\langle\delta_{x_0},C_{w,\psi}^m(f_M)\rangle|=\prod_{j=0}^{m-1}|w(\psi^j(x_0)) f_M(\psi^m(x_0))|\\
		&=\prod_{j=0}^{m-1}|w(\psi^j(x_0))||\delta_{\psi^m(x_0)}(f_M)|\geq \frac{M}{2}\prod_{j=0}^{m-1}|w(\psi^j(x_0))|
	\end{align*}
	which by $\prod_{j=0}^{m-1}|w(\psi^j(x_0))|\neq 0$ gives the desired contradiction.
	
	Thus,
	$$\psi^m\big(\{x\in X_{n-1};\,\prod_{j=0}^{m-1}|w(\psi^j(x))|\neq 0\}\big)\subseteq\overline{X}_l$$
	for every $m\in\N$. Since $X_{n-1}$ is open in $X$, $\psi^m$ is continuous, and $\{y\in X;\,\forall\,0\leq l\leq m:\,w(\psi^l(y))\neq 0\}$ is dense in $X$ by hypothesis c), we conclude
	$$\forall\,m\in\N:\,\psi^m(\overline{X}_{n-1})\subseteq\overline{X}_l.$$
	Because $(X_n)_{n\in\N}$ is an open, relatively compact exhaustion of $X$ it follows that $\psi$ has stable orbits.
\end{proof}

\begin{proposition}\label{stable sufficient}
	Assume that $\mathscr{F}$ satisfies $(\mathscr{F}1)$ and that the topology defining the sheaf $\mathscr{F}$ is the compact-open topology. Moreover, let $X\subseteq\Omega$ be open. If $\psi$ has stable orbits then $C_{w,\psi}$ is $m$-topologizable on $\mathscr{F}(X)$.
\end{proposition}

\begin{proof}
	Fix a compact $K\subseteq X$ and let $L\subseteq X$ be compact such that $\psi^m(K)\subseteq L$ for every $m\in\N$. Then, for every $f\in\mathscr{F}(X)$ and each $m\in\N$ it follows
	$$\|C_{w,\psi}^m(f)\|_K=\sup_{x\in K}|\prod_{j=0}^{m-1}w(\psi^j(x)) f(\psi^m(x))|\leq\|w\|_L^m\|f\|_L$$
	which shows the $m$-topologizability of $C_{w,\psi}$.
\end{proof}

Combining the preceding two propositions we see that $m$-topologizability and topologizability for weighted composition operators on sheaves $\mathscr{F}$ which are endowed with the compact-open topology are equivalent under some additional assumptions on the space $\mathscr{F}(X)$ and some mild additional conditions on the weight and the symbol. While in concrete situations condition b) in Proposition \ref{stable necessary} is probably easy to check, Proposition \ref{denseness result} below gives sufficient conditions on the weight and the symbol under which condition c) follows.

\begin{corollary}\label{characterizing topologizability for compact-open}
	Let $\mathscr{F}$ satisfy $(\mathscr{F}1)$ and assume that the topology defining the sheaf $\mathscr{F}$ is the compact-open topology. Moreover, let $X\subseteq\Omega$ be open and assume that additionally the following conditions hold.
	\begin{enumerate}
		\item[a)] There is an open, relatively compact exhaustion $(X_n)_{n\in\N}$ of $X$ such that for each $n\in\N$ and every $x\in X\backslash\overline{X}_n$ as well as every open $W\subseteq X\backslash\overline{X}_n$ containing $x$ there is $U\subseteq W$ open, $x\in U$ for which the restriction $r_X^{X_n\cup U}$ has dense range.
		\item[b)] $\mbox{ker }\delta_x\neq \mathscr{F}(X)$ for each $x\in X$.
		\item[c)] For every $m\in\N_0$ the set
		$$\{y\in X;\,w(\psi^m(y))\neq 0\}$$
		is dense in $X$.
	\end{enumerate}
	Then, the following are equivalent.
	\begin{enumerate}
		\item[i)] $C_{w,\psi}$ is $m$-topologizable on $\mathscr{F}(X)$.
		\item[ii)] $C_{w,\psi}$ is topologizable on $\mathscr{F}(X)$.
		\item[iii)] $\psi$ has stable orbits.
	\end{enumerate}
\end{corollary}

As we shall see now, Propositions \ref{stable necessary} and \ref{stable sufficient} also enable to characterize power boundedness of weighted composition operators in terms of the weight and the symbol in case the topology of the sheaf $\mathscr{F}$ is the compact-open topology.

\begin{theorem}\label{power boundedness for compact-open}
	Assume that $\mathscr{F}$ satisfies $(\mathscr{F}1)$ and that the topology defining the sheaf $\mathscr{F}$ is the compact-open topology. Let $X\subseteq\Omega$ be open and assume that $w\in\mathscr{F}(X)$ and that additionally the following conditions hold.
	\begin{enumerate}
		\item[a)] There is an open, relatively compact exhaustion $(X_n)_{n\in\N}$ of $X$ such that for each $n\in\N$ and every $x\in X\backslash\overline{X}_n$ as well as every open $W\subseteq X\backslash\overline{X}_n$ containing $x$ there is $U\subseteq W$ open, $x\in U$ for which the restriction $r_X^{X_n\cup U}$ has dense range.
		\item[b)] $\mbox{ker }\delta_x\neq \mathscr{F}(X)$ for each $x\in X$.
		\item[c)] For every $m\in\N_0$ the set
		$$\{y\in X;\,w(\psi^m(y))\neq 0\}$$
		is dense in $X$.
	\end{enumerate}
	Then, the following are equivalent.
	\begin{enumerate}
		\item[i)] $C_{w,\psi}$ is power bounded on $\mathscr{F}(X)$.
		\item[ii)] $C_{w,\psi}$ is $m$-topologizable on $\mathscr{F}(X)$ and $\{C_{w,\psi}^m(w);\,m\in\N\}$ is bounded in $\mathscr{F}(X)$.
		\item[iii)] $C_{w,\psi}$ is topologizable on $\mathscr{F}(X)$ and $\{C_{w,\psi}^m(w);\,m\in\N\}$ is bounded in $\mathscr{F}(X)$.
		\item[iv)] $\psi$ has stable orbits and $\{C_{w,\psi}^m(w);\,m\in\N\}$ is bounded in $\mathscr{F}(X)$.
	\end{enumerate}
\end{theorem}

\begin{proof}
	Since $w\in\mathscr{F}(X)$, i) trivially implies ii), and obviously ii) implies iii). If iii) holds, iv) follows from Proposition \ref{stable necessary}.
	
	If, on the other hand iv) is satisfied, let $K\subseteq X$ be compact and choose $L\subseteq X$ compact such that $\psi^m(K)\subseteq L$ for all $m\in\N$. Then, for every $f\in\mathscr{F}(X)$ and $m\in\N$ it follows
	$$\|C_{w,\psi}^m(f)\|_K=\sup_{x\in K}|\prod_{j=0}^{m-1}w(\psi^j(x))f(\psi^m(x))|\leq \|C_{w,\psi}^{m-1}(w)\|_K\|f\|_L.$$
	Since $\{C_{w,\psi}^m(w);\,m\in\N\}$ is bounded there is $c>0$ such that for all $m\in\N$ and $f\in\mathscr{F}(X)$
	$$\|C_{w,\psi}^m(f)\|_K\leq c\|f\|_L$$
	so that $C_{w,\psi}$ is power bounded.
\end{proof}

\begin{corollary}\label{mean ergodic for compact-open}
	Assume that $\mathscr{F}$ satisfies $(\mathscr{F}1)$, that the topology defining the sheaf $\mathscr{F}$ is the compact-open topology, and that $\mathscr{F}(X)$ is a Fr\'echet-Montel space for every $X\subseteq\Omega$ open. Assume, that additionally the following conditions hold.
	\begin{enumerate}
		\item[a)] There is an open, relatively compact exhaustion $(X_n)_{n\in\N}$ of $X$ such that for each $n\in\N$ and every $x\in X\backslash\overline{X}_n$ as well as every open $W\subseteq X\backslash\overline{X}_n$ containing $x$ there is $U\subseteq W$ open, $x\in U$ for which the restriction $r_X^{X_n\cup U}$ has dense range.
		\item[b)] $\mbox{ker }\delta_x\neq \mathscr{F}(X)$ for each $x\in X$.
	\end{enumerate}
	Then, the following are equivalent.
	\begin{enumerate}
		\item[i)] $C_\psi$ is power bounded on $\mathscr{F}(X)$.
		\item[ii)] $C_\psi$ is uniformly mean ergodic on $\mathscr{F}(X)$.
		\item[iii)] $C_\psi$ is mean ergodic on $\mathscr{F}(X)$.
		\item[iv)] $C_\psi$ is $m$-topologizable on $\mathscr{F}(X)$.
		\item[v)] $C_\psi$ is topologizable on $\mathscr{F}(X)$.
		\item[vi)] $\psi$ has stable orbits.
	\end{enumerate}
\end{corollary}

\begin{proof}
	Since Fr\'echet spaces are LF-spaces, i) implies ii) by \cite[Proposition 2.8]{AlbaneseBonetRicker09}.
	
	Trivially, ii) implies iii) and iii) implies v).
	
	By Corollary \ref{characterizing topologizability for compact-open} iv), v), and vi) are equivalent so that it remains to prove that vi) implies i).
	
	Let $K\subseteq X$ be compact. We choose $L\subseteq X$ compact such that $\psi^m(K)\subseteq L$ for every $m\in\N_0$ so that
	$$\forall\,f\in F(X), m\in\N_0:\,\|C_\psi^m(f)\|_K\leq \|f\|_L,$$
	i.e.\ $\{C_\psi^m(f);\,m\in\N_0\}$ is a bounded subset of $\mathscr{F}(X)$ for every $f\in\mathscr{F}(X)$. Since $\mathscr{F}(X)$ is a Fr\'echet space, hence barreled, i) follows. 
\end{proof}

The next proposition gives a sufficient condition for hypothesis c) in Corollary \ref{characterizing topologizability for compact-open} and Theorem \ref{power boundedness for compact-open} which is easily applicable in many concrete situations.

\begin{proposition}\label{denseness result}
	Assume that the following two conditions are satisfied.
	\begin{enumerate}
		\item[i)] $w^{-1}(\K\backslash\{0\})$ is dense in $X$.
		\item[ii)] For every $x\in X$ there is an open neighborhood $U_x$ of $x$ in $X$ such that $\psi_{|U_x}$ is injective and open.
	\end{enumerate}
	Then for every $m\in\N_0$ the set
	$$\{y\in X;\,w(\psi^m(y))\neq 0\}$$
	is dense in $X$.
\end{proposition}

\begin{proof}
	It easily follows from ii) that
	$$\forall\,x\in X, m\in\N\,\exists\, U_{x,m}\subseteq X\mbox{ open}, x\in U_{x,m}:\,\psi^m_{\;\;|U_{x,m}}\mbox{ injective and open}.$$
	
	Now, fix $x\in X$, $m\in\N_0$, and an open, non-empty subset $V$ of $U_{x,m}$. It follows from the injectivity of $\psi^m_{\;\;|U_{x,m}}$ that
	$$V\cap \big(w\circ\psi^m\big)^{-1}\big(\K\backslash\{0\}\big)=(\psi^m_{|U_{x,m}})^{-1}\big(\psi^m_{|U_{x,m}}(V)\cap w^{-1}(\K\backslash\{0\})\big)\neq \emptyset,$$
	since $w^{-1}(\K\backslash\{0\})$ is dense by hypothesis and $\psi^m_{|U_{x,m}}(V)$ is non-empty and open. Hence, $\big(w\circ\psi^m\big)^{-1}(\K\backslash\{0\})$ is open and dense in $U_{x,m}$. Since $x\in X$ was chosen arbitrarily, the claim follows.
\end{proof}

Next, we have a closer look at sheaves $\mathscr{F}$ on $\Omega=\R^d$ of $C^r$-functions ($r\in\N\cup\{\infty\}$) for which the defining topology is finer than the one induced by the family of seminorms $\{\|\cdot\|_{l,K};\,l<r+1, K\subseteq X\mbox{ compact}\}, X\subseteq\R^d$ open, where
$$\forall\,f\in C^r(X):\,\|f\|_{l,K}:=\sup_{|\alpha|\leq l}\sup_{x\in K}|\partial^\alpha f(x)|.$$
We start with a result about the derivatives of a composition. Recall that for the weight $w=1$ we simply write $C_\psi$ instead of $C_{1,\psi}$ (unweighted composition operator). For $X\subseteq\R^d$ open and $\psi:X\rightarrow X$ we denote the components of $\psi$ by $\psi_c, 1\leq c\leq d$.

\begin{proposition}\label{derivatives of compositions}
	Let $X\subseteq\R^d$ be open, $\psi:X\rightarrow X$ be a $C^r$-map, $r\in\N\cup\{\infty\}$.
	\begin{enumerate}
		\item[i)] For every $\alpha\in\N_0^d\backslash\{0\}$ with $|\alpha|< r+1$, every $m\in\N_0$, and for each $\beta\in\N_0^d$ with $1\leq|\beta|\leq|\alpha|$ there are $n(\beta)\in\N$ (with $n(\beta)=1$ whenever $|\beta|=1$) and $\gamma(\beta,k,j)\in\N_0^d$ for each such $\beta$, $1\leq k\leq n(\beta)$, $1\leq j\leq |\beta|$ with $\sum_{j=1}^{|\beta|}\gamma(\beta,k,j)=\alpha$ such that for every $f\in C^r(X)$,
		$$\partial^\alpha(C_\psi^m f)=\sum_{\beta\in\N_0^d, 1\leq|\beta|\leq|\alpha|}\big((\partial^\beta f)\circ\psi^m\big)\cdot\Big(\sum_{k=1}^{n(\beta)}\prod_{j=1}^{|\beta|}\partial^{\gamma(\beta,k,j)}\big(C^{m-1}_\psi\psi_{c(\beta,k,j)}\big)\Big),$$
		where $1\leq c(\beta,k,j)\leq d$. In particular $|\gamma(\beta,k,j)|\leq |\alpha|$.
		\item[ii)] For every $l\in\N_0, l<r+1,$ there is $M_l>0$ such that for every $K\subseteq X$ compact, every $m\in\N$, and each $f\in C^r(X)$ we have
		$$\|C_\psi^m f\|_{l,K}\leq \|f\|_{l,\psi^m(K)}M_l(\max_{1\leq c\leq d}\{1,\|C_\psi^{m-1}\psi_c\|_{l,K}\})^l.$$
	\end{enumerate}
\end{proposition}

\begin{proof}
	The proof of i) is done by induction on $|\alpha|$. If $|\alpha|=1$ then $\alpha=e_i$ for some $1\leq i\leq d$, where $e_i=(\delta_{l,i})_{1\leq l\leq d}$ denotes the canonical $i$-th standard basis vector in $\R^d$. Thus, for each $m\in\N_0$ we have
	$$\partial^\alpha(C^m_\psi f)=\partial_i(C^m_\psi f)=\sum_{c=1}^d\big((\partial_c f)\circ\psi^m\big)\cdot\partial_i(\psi^m)_c=\sum_{c=1}^d\big((\partial_c f)\circ\psi^m\big)\cdot\partial_i\big(C^{m-1}_\psi \psi_c\big),$$
	which proves the claim for $|\alpha|=1$. Now assume the claim is true for all $\alpha\in\N_0^d\backslash\{0\}$ with $|\alpha|\leq s$. Let $\alpha$ be of length $s$ and let $1\leq i\leq d$. By induction hypothesis we have
	\begin{align*}
	&\partial^{\alpha+e_i}(C_\psi^m f)=\partial_i\Big(\sum_{\beta\in\N_0^d, 1\leq|\beta|\leq|\alpha|}\big((\partial^\beta f)\circ\psi^m\big)\cdot\Big(\sum_{k=1}^{n(\beta)}\prod_{j=1}^{|\beta|}\partial^{\gamma(\beta,k,j)}\big(C^{m-1}_\psi\psi_{c(\beta,k,j)}\big)\Big)\Big)\\
	&=\sum_{\beta\in\N_0^d, 1\leq|\beta|\leq|\alpha|}\sum_{c=1}^d\big((\partial^{\beta+e_c}f)\circ\psi^m\big)\cdot\partial_i\big(C_\psi^{m-1}\psi_c\big)\cdot\Big(\sum_{k=1}^{n(\beta)}\prod_{j=1}^{|\beta|}\partial^{\gamma(\beta,k,j)}\big(C^{m-1}_\psi\psi_{c(\beta,k,j)}\big)\Big)\\
	&\quad +\sum_{\beta\in\N_0^d, 1\leq|\beta|\leq|\alpha|}\big((\partial^\beta f)\circ\psi^m\big)\cdot\Big(\sum_{k=1}^{n(\beta)}\partial_i\big(\prod_{j=1}^{|\beta|}\partial^{\gamma(\beta,k,j)}\big(C_\psi^{m-1}\psi_{c(\beta,k,j)}\big)\big)\Big)\\
	&=\sum_{\beta\in\N_0^d, |\beta|=|\alpha|}\sum_{c=1}^d\big((\partial^{\beta+e_c}f)\circ\psi^m\big)\cdot\Big(\sum_{k=1}^{n(\beta)}\partial_i\big(C_\psi^{m-1}\psi_c\big)\cdot\prod_{j=1}^{|\beta|}\partial^{\gamma(\beta,k,j)}\big(C_\psi^{m-1}\psi_{c(\beta,k,j)}\big)\Big)\\
	&\quad +\sum_{\beta\in\N_0^d, 1\leq|\beta|\leq|\alpha|-1}\sum_{c=1}^d\big((\partial^{\beta+e_c}f)\circ\psi^m\big)\\
	&\quad\cdot\Big(\sum_{k=1}^{n(\beta)}\partial_i\big(C_\psi^{m-1}\psi_c\big)\cdot\prod_{j=1}^{|\beta|}\partial^{\gamma(\beta,k,j)}\big(C_\psi^{m-1}\psi_{c(\beta,k,j)}\big)\Big)\\
	&\quad +\sum_{\beta\in\N_0^d, 2\leq|\beta|\leq|\alpha|}\big((\partial^\beta f)\circ\psi^m\big)\\
	&\quad\cdot\Big(\sum_{k=1}^{n(\beta)}\sum_{p=1}^{|\beta|}\big(\prod_{j=1,j\neq p}^{|\beta|}\partial^{\gamma(\beta,k,j)}\big(C_\psi^{m-1}\psi_{c(\beta,k,j)}\big)\big)  \cdot\partial^{\gamma(\beta,k,j)+e_i}\big(C_\psi^{m-1}\psi_{c(\beta,k,p)}\big)\Big)
\end{align*}
\begin{align*}
	&\quad +\sum_{\beta\in\N_0^d,|\beta|=1}\big((\partial^\beta f)\circ\psi^m\big)\cdot\partial^{\alpha+e_i}\big(C_\psi^{m-1}\psi_{c(\beta,1,1)}\big)\\
	&=\sum_{\tilde{\beta}\in\N_0^d, |\tilde{\beta}|=|\alpha|+1}\big((\partial^{\tilde{\beta}}f)\circ\psi^m\big)\\
	&\quad\cdot\Big(\sum\limits_{\substack{1\leq c\leq d, \beta\in\N_0^d,\\ |\beta|=|\alpha|, \beta+e_c=\tilde{\beta}}}\sum_{k=1}^{n(\beta)}\partial_i\big(C_\psi^{m-1}\psi_c\big)\cdot\prod_{j=1}^{|\beta|}\partial^{\gamma(\beta,k,j)}\big(C_\psi^{m-1}\psi_{c(\beta,k,j)}\big)\Big)\\
	&\quad +\sum_{k=2}^{|\alpha|}\sum_{\tilde{\beta}\in\N_0^d, |\tilde{\beta}|=k}\big((\partial^{\tilde{\beta}}f)\circ\psi^m\big)\\
	&\quad\cdot\Big(\sum\limits_{\substack{1\leq c\leq d, \beta\in\N_0^d,\\ |\beta|=|\alpha|, \beta+e_c=\tilde{\beta}}}\sum_{k=1}^{n(\beta)}\partial_i\big(C_\psi^{m-1}\psi_c\big)\cdot\prod_{j=1}^{|\beta|}\partial^{\gamma(\beta,k,j)}\big(C_\psi^{m-1}\psi_{c(\beta,k,j)}\big)\\
	&\quad +\sum_{k=1}^{n(\tilde{\beta})}\sum_{p=1}^{|\tilde{\beta}|}\big(\prod_{j=1,j\neq p}^{|\tilde{\beta}|}\partial^{\gamma(\tilde{\beta},k,j)}\big(C_\psi^{m-1}\psi_{c(\tilde{\beta},k,j)}\big)\big)\cdot\partial^{\gamma(\tilde{\beta},k,p)+e_i}\big(C_\psi^{m-1}\psi_{c(\tilde{\beta},k,p)}\big)\Big)\\
	&\quad +\sum_{\tilde{\beta}\in\N_0^d, |\tilde{\beta}|=1}\big((\partial^{\tilde{\beta}}f)\circ\psi^m\big)\cdot\partial^{\alpha+e_i}\big(C_\psi^{m-1}\psi_{c(\tilde{\beta},1,1)}\big)\\
	&=\sum_{\tilde{\beta}\in\N_0^d, |\tilde{\beta}|=|\alpha|+1}\big((\partial^{\tilde{\beta}}f)\circ\psi^m\big)\\
	&\quad\cdot\Big(\sum\limits_{\substack{1\leq c\leq d, \beta\in\N_0^d,\\ |\beta|=|\alpha|, \beta+e_c=\tilde{\beta}}}\sum_{k=1}^{n(\beta)}\partial^{e_i}\big(C_\psi^{m-1}\psi_c\big)\cdot\prod_{j=1}^{|\beta|}\partial^{\gamma(\beta,k,j)}\big(C_\psi^{m-1}\psi_{c(\beta,k,j)}\big)\Big)\\
	&\quad +\sum_{k=2}^{|\alpha|}\sum_{\tilde{\beta}\in\N_0^d, |\tilde{\beta}|=k}\big((\partial^{\tilde{\beta}}f)\circ\psi^m\big)\\
	&\quad\cdot\Big(\sum\limits_{\substack{1\leq c\leq d, \beta\in\N_0^d,\\ |\beta|=|\alpha|, \beta+e_c=\tilde{\beta}}}\sum_{k=1}^{n(\beta)}\partial^{e_i}\big(C_\psi^{m-1}\psi_c\big)\cdot\prod_{j=1}^{|\beta|}\partial^{\gamma(\beta,k,j)}\big(C_\psi^{m-1}\psi_{c(\beta,k,j)}\big)\\
	&\quad +\sum_{k=1}^{n(\tilde{\beta})}\sum_{p=1}^{|\tilde{\beta}|}\big(\prod_{j=1,j\neq p}^{|\tilde{\beta}|}\partial^{\gamma(\tilde{\beta},k,j)}\big(C_\psi^{m-1}\psi_{c(\tilde{\beta},k,j)}\big)\big)\cdot\partial^{\gamma(\tilde{\beta},k,p)+e_i}\big(C_\psi^{m-1}\psi_{c(\tilde{\beta},k,p)}\big)\Big)\\
	&\quad +\sum_{\tilde{\beta}\in\N_0^d,|\tilde{\beta}|=1}\big((\partial^{\tilde{\beta}}f)\circ\psi^m\big)\cdot\partial^{\alpha+e_i}\big(C_\psi^{m-1}\psi_{c(\tilde{\beta},1,1)}\big).
	\end{align*}
	Since by induction hypothesis we have $\sum_{j=1}^{|\beta|}\gamma(\beta,k,j)=\alpha$ it follows that $e_i+\sum_{j=1}^{|\beta|}\gamma(\beta,k,j)=\alpha+e_i$ as well as $\sum_{j=1,j\neq p}^{|\tilde{\beta}|}\gamma(\tilde{\beta},k,j)+\gamma(\tilde{\beta},k,p)+e_i=\alpha+e_i$. Hence the claim is also true for every multi-index of length equal to $s+1$ which proves the induction step. Thus, i) is proved.
	
	In order to prove ii), let $K\subseteq X$ be compact, $m\in\N$, and $x\in K$. For $\alpha\in\N_0^d\backslash\{0\}$ and $f\in C^r(X)$ it follows from i), taking into account that $|\gamma(\beta,k,j)|\leq|\alpha|$, with
	$$B_{|\alpha|}:=\begin{cases}
	|\alpha|\max_{1\leq|\beta|\leq|\alpha|}n(\beta) &\mbox{if }d=1\\ d^{|\alpha|}\max_{1\leq|\beta|\leq|\alpha|}n(\beta) &\mbox{if }d>1
	\end{cases}$$
	that
	\begin{align*}
		|\partial^\alpha\big(C_\psi ^m f\big)(x)| &\leq\sum_{\substack{\beta\in\N_0^d,\\ 1\leq|\beta|\leq|\alpha|}}|(\partial^\beta f)(\psi^m(x))|\cdot\Big(\sum_{k=1}^{n(\beta)}\prod_{j=1}^{|\beta|}|\partial^{\gamma(\beta,k,j)}\big(C_\psi^{m-1}\psi_{c(\beta,k,j)}\big)(x)|\Big)\\
		&\leq \|f\|_{|\alpha|,\psi^m(K)}\Big(\sum_{\substack{\beta\in\N_0^d,\\ 1\leq|\beta|\leq|\alpha|}}\sum_{k=1}^{n(\beta)}\prod_{j=1}^{|\beta|}\|C_\psi^{m-1}\psi_{c(\beta,k,j)}\|_{|\gamma(\beta,k,j)|,K}\Big)\\
		&\leq \|f\|_{|\alpha|,\psi^m(K)}\Big(\sum_{\substack{\beta\in\N_0^d,\\ 1\leq|\beta|\leq|\alpha|}}\max_{1\leq|\beta|\leq|\alpha|}n(\beta)\big(\max_{1\leq c\leq d}\{1,\|C_\psi^{m-1}\psi_c\|_{|\alpha|,K}\}\big)^{|\alpha|}\Big)\\
		&\leq \|f\|_{|\alpha|,\psi^m(K)}B_{|\alpha|}\big(\max_{1\leq c\leq d}\{1,\|C_\psi^{m-1}\psi_c\|_{|\alpha|,K}\}\big)^{|\alpha|}.
	\end{align*}
	Since an analogous inequality is obviously valid for $\alpha=0$, ii) follows.
\end{proof}

\begin{proposition}\label{stable sufficient in c^r}
	Let $\mathscr{F}$ be a sheaf on $\R^d$ of $C^r$-functions, $r\in\N\cup\{\infty\}$, such that the defining topology of $\mathscr{F}$ is the $C^r$-topology. Moreover, let $X\subseteq\R^d$ be open and assume that $w$ as well as $\psi$ are $C^r$-functions. If $\psi$ has stable orbits, then $C_{w,\psi}$ is topologizable on $\mathscr{F}(X)$.
\end{proposition}

\begin{proof}
	Let $K\subseteq X$ be compact and $l\in\N_0$, $l<r+1$. Let $L\subseteq X$ be compact such that $\psi^m(K)\subseteq L$ for every $m\in\N$. Then, using Proposition \ref{derivatives of compositions} ii) we have for suitable $M_l$ independent of $K,L$, for every $f\in\mathscr{F}(X)$ and each $m\in\N$,
	\begin{align*}
		\|C_{w,\psi}^m (f)\|_{l,K}&\leq\sup_{|\alpha|\leq l}\sup_{x\in K}\sum_{\beta\leq\alpha}{\alpha\choose\beta}|\partial^\beta\big(C_{w,\psi}^{m-1}(w)\big)(x)||\partial^{\alpha-\beta}\big(C_\psi^m (f)\big)(x)|\\
		&\leq \big(\max_{|\alpha|\leq l}\sum_{\beta\leq\alpha}{\alpha\choose\beta}\big)\|C_{w,\psi}^{m-1}(w)\|_{l,K}\|C_\psi^m(f)\|_{l,K}\\
		&\leq\Big(2^lM_l\|C_{w,\psi}^{m-1}(w)\|_{l,K}(\max_{1\leq c\leq d}\{1,\|C_\psi^{m-1}(\psi_c)\|_{l,K}\})^l\Big)\|f\|_{l,L}.
	\end{align*}
\end{proof}

Combining Propositions \ref{stable necessary} and \ref{stable sufficient in c^r} we obtain the following.

\begin{corollary}\label{characterizing topologizabilty for smooth}
	Let $\mathscr{F}$ satisfy $(\mathscr{F}1)$ be a sheaf on $\R^d$ of $C^r$-functions, $r\in\N\cup\{\infty\}$ such that the defining topology of $\mathscr{F}$ is the $C^r$-topology. Moreover, let $X\subseteq\R^d$ be open, $w, \psi$ be $C^r$-functions and assume that additionally the following conditions hold.
	\begin{enumerate}
		\item[a)] There is an open, relatively compact exhaustion $(X_n)_{n\in\N}$ of $X$ such that for each $n\in\N$ and every $x\in X\backslash\overline{X}_n$ as well as every open $W\subseteq X\backslash\overline{X}_n$ containing $x$ there is $U\subseteq W$ open, $x\in U$ for which the restriction $r_X^{X_n\cup U}$ has dense range.
		\item[b)] $\mbox{ker }\delta_x\neq \mathscr{F}(X)$ for each $x\in X$.
		\item[c)] For every $m\in\N_0$ the set
		$$\{y\in X;\,w(\psi^m(y))\neq 0\}$$
		is dense in $X$.
	\end{enumerate}
	Then, the following are equivalent.
	\begin{enumerate}
		\item[i)] $C_{w,\psi}$ is topologizable on $\mathscr{F}(X)$.
		\item[ii)] $\psi$ has stable orbits.
	\end{enumerate}
\end{corollary}

In the case of a sheaf of $C^r$-functions equipped with the $C^r$-topology we are now ready to give a characterization of power boundedness for weighted composition operators.

\begin{theorem}\label{power boundedness for smooth}
	Let $\mathscr{F}$ satisfy $(\mathscr{F}1)$ be a sheaf on $\R^d$ of $C^r$-functions, $r\in\N\cup\{\infty\}$ such that the defining topology of $\mathscr{F}$ is the $C^r$-topology. Moreover, let $X\subseteq\R^d$ be open and $w, \psi$ be such that $w, \psi_c\in\mathscr{F}(X)$ for all $1\leq c\leq d$ and assume that additionally the following conditions hold.
	\begin{enumerate}
		\item[a)] There is an open, relatively compact exhaustion $(X_n)_{n\in\N}$ of $X$ such that for each $n\in\N$ and every $x\in X\backslash\overline{X}_n$ as well as every open $W\subseteq X\backslash\overline{X}_n$ containing $x$ there is $U\subseteq W$ open, $x\in U$ for which the restriction $r_X^{X_n\cup U}$ has dense range.
		\item[b)] $\mbox{ker }\delta_x\neq \mathscr{F}(X)$ for each $x\in X$.
		\item[c)] For every $m\in\N_0$ the set
		$$\{y\in X;\,w(\psi^m(y))\neq 0\}$$
		is dense in $X$.
	\end{enumerate}
	Then, the following are equivalent.
	\begin{enumerate}
		\item[i)] $C_{w,\psi}$ is power bounded on $\mathscr{F}(X)$.
		\item[ii)] $C_{w,\psi}$ is topologizable  on $\mathscr{F}(X)$, $\{C_{w,\psi}^m(w);\,m\in\N\}$  and $\{C_{w,\psi}^m(\psi_c);\,m\in\N\}$, $1\leq c\leq d$, are bounded in $\mathscr{F}(X)$.
		\item[iii)] $\psi$ has stable orbits and $\{C_{w,\psi}^m(w);\,m\in\N\}$  as well as $\{C_{w,\psi}^m(\psi_c);\,m\in\N\}$, $1\leq c\leq d$, are bounded in $\mathscr{F}(X)$.
	\end{enumerate}
\end{theorem}

\begin{proof}
	Clearly, since $w\in\mathscr{F}(X)$ and $\psi_c\in\mathscr{F}(X)$ for every $1\leq c\leq d$, i) implies ii). Moreover, ii) implies iii) by Corollary \ref{characterizing topologizabilty for smooth}. Finally, if iii) holds, it follows as in the proof of Proposition \ref{stable sufficient in c^r} that for every $l\in\N_0$ there is $M_l>0$ such that for every compact $K,L \subseteq X$ with $\psi^m(K)\subseteq L$ for each $m\in\N_0$ we have 
	\begin{align*}
	\|C_{w,\psi}^m (f)\|_{l,K}&\leq2^l M_l \|C_{w,\psi}^{m-1}(w)\|_{l,K} (\max_{1\leq c\leq d}\{1,\|C_\psi^{m-1}(\psi_c)\|_{l,K}\})^l\|f\|_{l,L}
	\end{align*}
	for every $f\in\mathscr{F}(X)$. Because $\{C_{w,\psi}^m(w);\,m\in\N\}$  as well as $\{C_{w,\psi}^m(\psi_c);\,m\in\N\}$ are bounded in $\mathscr{F}(X)$ for every $1\leq c\leq d$, this inequality implies the existence of $M>0$ such that for every $m\in\N$ and $f\in \mathscr{F}(X)$
	$$\|C_{w,\psi}^m(f)\|_{l,K}\leq M\|f\|_{l,L},$$
	so that i) follows.
\end{proof}

Before we close this section we give some first applications of our results to concrete sheaves. We begin with the sheaf of continuous functions. By Brouwer's Invariance of Domain Theorem, the next result is in particular applicable in the case of $\Omega=\R^d$ and $\psi$ being locally injective.

\begin{corollary}\label{continuous functions}
	Let $\Omega$ be a locally compact, $\sigma$-compact, non-compact Hausdorff space, $X\subseteq\Omega$ open, $w\in C(X)$ and $\psi:X\rightarrow X$ be continuous such that $w^{-1}(\K\backslash\{0\})$ is dense in $X$ and such that for every $x\in X$ there is an open neighborhood $U_x$ of $x$ in $X$ such that $\psi_{|U_x}$ is injective and open. Then, the following are equivalent.
	\begin{enumerate}
		\item[i)] $C_{w,\psi}$ is power bounded on $C(X)$.
		\item[ii)] $C_{w,\psi}$ is $m$-topologizable (or topologizable) on $C(X)$ and $\{C_{w,\psi}^m(w);\,m\in\N\}$ is bounded in $C(X)$.
		\item[iii)] $\psi$ has stable orbits and $\{C_{w,\psi}^m(w);\,m\in\N\}$ is bounded in $C(X)$.
	\end{enumerate}
\end{corollary}

\begin{proof}
	While hypothesis b) of Theorem \ref{power boundedness for compact-open} is obviously satisfied, hypothesis c) holds by Proposition \ref{denseness result}. Moreover, applying Tietze's Extension Theorem, it follows that for every open, relatively compact exhaustion $(X_n)_{n\in\N}$ of $X$, for each $n\in\mathbb{N}$ and every $x\in X\backslash\overline{X}_n$ and every open neighborhood $U$ of $x$ with $U\subseteq X\backslash\overline{X}_n$ the restriction map $r_X^{X_n\cup U}$ has dense range. Therefore, hypothesis a) of Theorem \ref{power boundedness for compact-open} also holds so that Theorem \ref{power boundedness for compact-open} implies the claim.
\end{proof}

We continue with the sheaf of smooth functions on $\Omega=\R^d$.

\begin{corollary}\label{smooth functions}
	Let $X\subseteq\R^d$ be open, $w\in C^\infty(X)$, and $\psi:X\rightarrow X$ be smooth and locally injective such that $w^{-1}(\K\backslash\{0\})$ is dense in $X$. Then, the following are equivalent.
	\begin{enumerate}
		\item[i)] $C_{w,\psi}$ is power bounded on $C^\infty(X)$.
		\item[ii)] $C_{w,\psi}$ is topologizable on $C^\infty(X)$, $\{C_{w,\psi}^m(w);\,m\in\N\}$ and $\{C_{w,\psi}^m(\psi_c);\,m\in\N\}$ are bounded in $C^\infty(X)$ for every $1\leq c\leq d$.
		\item[iii)] $\psi$ has stable orbits and $\{C_{w,\psi}^m(w);\,m\in\N\}$ as well as $\{C_{w,\psi}^m(\psi_c);\,m\in\N\}$ are bounded in $C^\infty(X)$ for every $1\leq c\leq d$.
	\end{enumerate}
\end{corollary}

\begin{proof}
	Again, hypothesis b) of Theorem \ref{power boundedness for smooth} is obviously satisfied while hypothesis c) holds by Brouwer's Invariance of Domain Theorem and Proposition \ref{denseness result}. Let $(X_n)_{n\in\N}$ be an arbitrary open, relatively compact exhaustion of $X$, $x\in X\backslash\overline{X}_n$, $U$ a neighborhood of $x$ in $X\backslash\overline{X}_n$, and $f\in C^\infty(X_n\cup U)$. If $K\subseteq X_n\cup U$ is compact, let $\varphi$ be a smooth function on $\R^d$ with support in $X_n\cup U$ which is equal to 1 in a neighborhood of $K$. Extending $\varphi f$ to $\R^d$ by zero outside $X_n\cup U$ we obtain a smooth function $g$ on $\R^d$ such that $\|g-f\|_{K,l}=0$ for every $l\in\N$. In particular the restriction map $r_X^{X_n\cup U}$ has dense range so that hypothesis a) of Theorem \ref{power boundedness for smooth} is fulfilled, too. Thus, the claim follows immediately from Theorem \ref{power boundedness for smooth}.
\end{proof}

As a final example in this section we consider the sheaf of holomorphic functions. The corresponding result for $d=1$ was originally proved in \cite[Theorem 3.3]{BeltranGomezJordaJornet16} while the general case was alluded to in \cite[Remark 3.8]{BeltranGomezJordaJornet16}. The special case of $w=1$ has been proved in \cite{BonetDomanski11}.

\begin{corollary}\label{holomorphic functions}
	Let $X\subseteq\C^d$ be a domain of holomorphy, let $w:X\rightarrow\C$ and $\psi:X\rightarrow X$ be holomorphic such that $\{z\in X;\,w(\psi^m(z))\neq 0\}$ is dense in $X$ for every $m\in\N_0$. Then the following are equivalent.
	\begin{enumerate}
		\item[i)] $C_{w,\psi}$ is power bounded on $\mathscr{H}(X)$.
		\item[ii)] $C_{w,\psi}$ is ($m$-)topologizable on $\mathscr{H}(X)$ and $\{C_{w,\psi}^m(w);\,m\in\N\}$ is bounded in $\mathscr{H}(X)$.
		\item[iii)] $\psi$ has stable orbits and $\{C_{w,\psi}^m(w);\,m\in\N\}$ is bounded in $\mathscr{H}(X)$. 
	\end{enumerate}
	For the special case that $w=1$ the above are further equivalent to
	\begin{enumerate}
		\item[iv)] $C_\psi$ is (uniformly) mean ergodic on $\mathscr{H}(X)$.
	\end{enumerate}
\end{corollary}

\begin{proof}
	While hypothesis c) of Theorem \ref{power boundedness for compact-open} is satisfied by assumption, hypothesis b) is obviously satisfied. In order to see that hypothesis a) of Theorem \ref{power boundedness for compact-open} is also fulfilled, we first recall that $X$ is pseudoconvex (see e.g.\ \cite[Theorem 4.2.8]{Hoermander_Complex}), thus there is a continuous plurisubharmonic $u:X\rightarrow\R$ such that
	$$\forall\,c\in\R:\,X_c:=\{z\in X;\,u(z)<c\}\mbox{ is a relatively compact subset of }X,$$
	i.e.\ $(X_n)_{n\in\N}$ is a relatively compact, open exhaustion of $X$. Obviously, for each $c\in\R$
	$$K_c:=\{z\in X;\,u(z)\leq c\}$$
	is a compact subset of $X$ and by \cite[Theorem 4.3.4]{Hoermander_Complex} $\hat{K}_{c,X}^P=\hat{K}_{c,X}=K_c$, where
	$$\hat{K}_{c,X}^P=\{z\in X;\,v(z)\leq \sup_{y\in K_c}v(y)\mbox{ for all plurisubharmonic }v\mbox{ on }X\}$$
	and
	$$\hat{K}_{c,X}=\{z\in X;\,|f(z)|\leq \sup_{y\in K_c}|f(y)|\mbox{ for all }f\in\mathscr{H}(X)\}$$
	(observe also that $\hat{K}_{c,X}^P\subseteq K_c$). Thus, $K_c=\hat{K}_c,$ and then, $K_c$ is holomorphically convex, where for compact subsets $K\subseteq\C^l$ we simply write $\widehat{K}$ for the holomorphically convex hull of $K$ in $\C^l$.
	
	Fix $c\in\R$ and $x_0\in X\backslash K_c$. Hence there is $f\in \mathscr{H}(X)$ such that
	$$|f(x_0)|>\sup_{y\in K_c}|f(y)|=:r.$$
	Denoting the closed ball in $\C$ with center $z\in\C$ and radius $\rho>0$ by $B_\C[z,\rho]$ and the closed polydisk in $\C^d$ with center $z\in\C^d$ and polyradius $\rho$ by $PB_{\C^d}(z,\rho)$ it follows that $\widehat{f(K_c)}\subseteq B_\C[0,r]$ and that there are $\delta_c,\varepsilon>0$ with
	$$f(PB_{\C^d}[x_0,\delta_c])\subseteq B_\C[f(x_0),\varepsilon]\cap\big(\C\backslash B_\C[0,r+\varepsilon]\big)$$
	so that
	$$\forall\,\delta\in (0,\delta_c):\,\emptyset=\reallywidehat{f(PB_{\C^d}[x_0,\delta])}\cap \widehat{f(K_c)}.$$ 
	Since (poly)disks are holomorphically convex, it follows from the above together with the Kallin lemma (cf.\ \cite[Lemma 2]{Zajac}) that
	$$\forall\,\delta\in(0,\delta_c):\,M_\delta:=PB_{\C^d}[x_0,\delta]\cup K_c$$
	is holomorphically convex. Therefore, by the Oka-Weil Theorem (see e.g.\cite[Corollary 5.2.9]{Hoermander_Complex}) it follows that every function holomorphic on a neighborhood of $M_\delta$ can be approximated uniformly on $M_\delta$ by functions in $\mathscr{H}(\C^d)$.
	
	Since $(K_{n-\frac{1}{m}})_{m\in\N}$ is a compact exhaustion on $X_n, n\in\N,$ it follows that for $x\in X\backslash\overline{X}_n$ and an open neighborhood $W\subseteq X\backslash\overline{X}_n$ of $x_0$ the restriction mapping $r_X^{X_n\cup PB_{\C^d}(x_0,\rho_n)}$ has dense range, where $\rho_n\in (0,\delta_n)$ is such that $PB_{\C^d}(x_0,\rho_n)\subseteq W$. Hence hypothesis a) of Theorem \ref{power boundedness for compact-open} is indeed satisfied so that the claim follows from this theorem.
	
	In case of $w=1$ the claim follows from Corollary \ref{mean ergodic for compact-open}.
\end{proof}

\section{Weighted composition operators on kernels of differential operators}\label{kernels of differential operators}

In this section we apply the results from section \ref{power boundedness} to weighted composition operators defined on kernels of partial differential operators on spaces of smooth functions. The special case of the Cauchy-Riemann operator will give the space of holomorphic functions of a single variable equipped with the compact-open topology. In this context topologizability and power boundedness of weighted composition operators have been studied in \cite{BeltranGomezJordaJornet16}.

As already mentioned in example \ref{examples of sheaves} iv), for a non-constant polynomial with complex coefficients in $d\geq 2$ variables $P\in\C[X_1,\ldots,X_d]$ and an open subset $X\subseteq\R^d$ we define
$$C_P^{\infty}(X):=\{u\in C^\infty(X);\,P(\partial)u=0\mbox{ in }X\},$$
where for $P(\xi)=\sum_{|\alpha|\leq m}a_\alpha\xi^\alpha$ with $a_{\alpha_0}\neq 0$ for some multi-index $\alpha_0\in\N_0^d$ with $|\alpha_0|(=\alpha_1+\ldots+\alpha_d)=m$ we define
$$\forall\,u\in C^\infty(X), x\in X:\,P(\partial)u(x)=\sum_{|\alpha|\leq m}a_\alpha\partial^\alpha u(x).$$
We denote by $P_m(\xi):=\sum_{|\alpha|=m}a_\alpha\xi^\alpha$ the principal part of $P$. Recall that $P$ is called elliptic provided that $P_m(\xi)\neq 0$ for all $\xi\in\R^d\backslash\{0\}$. 

As a closed subspace of the separable nuclear Fr\'echet space $C^\infty(X)$ the space $C_P^\infty(X)$ is again a separable nuclear Fr\'echet space. For hypoelliptic polynomials $P$ - by definition - for every open $X\subseteq\R^d$ the spaces $C_P^\infty(X)$ and
$$\mathscr{D}'_P(X):=\{u\in\mathscr{D}'(X);\,P(\partial)u=0\mbox{ in }X\}$$
coincide (that is, every distribution $u$ on $X$ which satisfies $P(\partial)u=0$ in $X$ is already a smooth function). By a result of Malgrange for hypoelliptic $P$ the spaces $C_P^\infty(X)$ and $\mathscr{D}'_P(X)$ also coincide as locally convex spaces when the latter is endowed with the relative topology inherited from $\mathscr{D}'(X)$ equipped with the strong dual topology as the topological dual of $\mathscr{D}(X)$. This implies in particular, that for hypoelliptic polynomials the compact-open topology on $C^\infty_P(X)$ and the relative topology inherited from $C^\infty(X)$ coincide. Therefore, for hypoelliptic polynomials $P$ the space $C_P^\infty(X)$ endowed with the compact-open topology is a separable nuclear, in particular Montel, Fr\'echet space. For a recent elegant proof of Malgrange's result - and further results about topological properties of $\mathscr{D}'_P(X)$ for arbitrary $P$ - we refer the reader to \cite{Wengenroth2}.

Recall that elliptic polynomials are hypoelliptic (see e.g.\ \cite[Theorem 11.1.10]{Hoermander}). Thus, the space of holomorphic functions $\mathscr{H}(X)$ over an open $X\subseteq\C$ equipped with the compact-open topology coincides (topologically) with $C_P^\infty(X)$ for $P(\xi)=\frac{1}{2}(\xi_1+i\xi_2)$. 

\begin{remark}\label{conditions from previous section}
	\begin{rm}
	As already mentioned in example \ref{examples of sheaves} iv), $C_P^\infty$ defines a sheaf on $\R^d$ which satisfies $(\mathscr{F}1)$. Moreover, for $\zeta\in\C^d$ with $P(\zeta)=0$ it follows immediately that
	$$e_\zeta:\R^d\rightarrow\C, x\mapsto\exp(\sum_{j=1}^d\zeta_j x_j)$$
	belongs to $C_P^\infty(X)$, so that $\mbox{ker}\,\delta_x\neq C_P^\infty(X)$ for every $x\in X$. Therefore, condition b) from Theorems \ref{power boundedness for compact-open} and \ref{power boundedness for smooth} are always satisfied for $C_P^\infty$ whenever $P$ is non-constant. Sufficient conditions on the weight $w$ and the symbol $\psi$ implying condition c) of both theorems are given in Proposition \ref{denseness result} so that only condition a) about the dense range of the restriction maps $r_X^{X_n\cup U}$ for a suitable open, relatively compact exhaustion $(X_n)_{n\in\N}$ - and certain (usually small) open sets $U\subseteq X\backslash X_n$ has to be checked.

	Therefore, in order to apply our results from the previous section we have a look at when restriction mappings between $C_P^\infty$-spaces have dense range. Recall that $P(\partial)$ is surjective on $C^\infty(X)$ precisely when $X$ is $P$-convex for supports, i.e.\ when for every compact $K\subseteq X$ there is a compact $L\subseteq X$ such that for every $u\in\mathscr{E}'(X)$ with $\supp\check{P}(\partial)u\subseteq K$ it follows $\supp u\subseteq L$, where as usual $\mathscr{E}'(X)$ denotes the space of distributions on $\R^d$ having compact support contained in $X$ and where $\check{P}(\xi)=P(-\xi)$ (see, e.g.\ \cite[Section 10.6]{Hoermander}).
	
	Using the gluing property of a sheaf it follows in particular from condition a) of Theorems \ref{power boundedness for compact-open} and \ref{power boundedness for smooth}, that for a suitable open, relatively compact exhaustion $(X_n)_{n\in\N}$ of $X$ the restriction maps $r_X^{X_n}$ from $C_P^\infty(X)$ to $C_P^\infty(X_n)$ have dense range. We will shortly see that this is equivalent to $P(\partial)$ being surjective on $C^\infty(X)$. While the latter is true for every elliptic partial differential operator (see e.g.\ \cite[Section 10.6 and Corollary 10.8.2]{Hoermander}), for arbitrary (hypoelliptic) partial differential operators this is not true in general.
	\end{rm}
\end{remark}

\begin{theorem}\label{good exhaustion}
	Let $P\in\C[X_1,\ldots,X_d]\backslash\{0\}$ and let $X\subseteq\R^d$ be open. Then the following are equivalent.
	\begin{itemize}
		\item[i)] $X$ is $P$-convex for supports.
		\item[ii)] For the open, relatively compact exhaustion $(X_n)_{n\in\N}$ of $X$ defined by
		$$X_n:=\{x\in X;\,|x|<n\mbox{ and }\dist(x,\R^d\backslash X)>\frac{1}{n}\},$$
		the restriction maps
		$$r_X^{X_n}:C_P^\infty(X)\rightarrow C_P^\infty(X_n), f\mapsto f_{|X_n}$$
		have dense range, where $\dist(x,\R^d\backslash X)$ denotes the euclidean distance from $x$ to $\R^d\backslash X$ with the convention $\dist(x,\emptyset)=\infty$.
		\item[iii)] There is an open, relatively compact exhaustion $(X_n)_{n\in\N}$ of $X$ such that the restriction maps
		$$r_X^{X_n}:C_P^\infty(X)\rightarrow C_P^\infty(X_n), f\mapsto f_{|X_n}$$
		have dense range.
	\end{itemize}
\end{theorem}

\begin{proof}
	We assume that i) holds. For $n\in\N$ let $X_n$ be defined as in ii) of the theorem. Obviously, $X_n$ is an open, relatively compact subset of $X$ such that $\overline{X}_n\subseteq X_{n+1}$ and $X=\cup_{n\in\N}X_n$. Since $X$ is $P$-convex for supports, for every $u\in\mathscr{E}'(X)$ we have $\dist(\supp u,\R^d\backslash X)=\dist(\supp \check{P}(\partial)u,\R^d\backslash X)$ by \cite[Theorem 10.6.3]{Hoermander} and since for every $u\in\mathscr{E}'(\R^d)$ the convex hulls of the sets $\supp u$ and $\supp \check{P}(\partial)u$ coincide (cf.\ \cite[Theorem 7.3.2]{Hoermander}) we conclude that $X_n$ is $P$-convex for supports, too, as well as
	\begin{equation}\label{Grothendieck-Koethe preparation}
		\forall\,n\in\N, u\in\mathscr{E}'(X)\,: \supp \check{P}(\partial)u\subseteq X_n \Rightarrow \supp u\subseteq X_n.
	\end{equation}
	An application of \cite[Theorem 26.1]{Treves} (see also \cite[Theorem 4]{Kalmes18}) now yields ii).
	
	Obviously, ii) implies iii).
	
	Next, we assume that iii) is satisfied. It follows from \cite[Theorem 3.2.8]{Wengenroth} that
	$$0=\mbox{Proj}^1\big(C_P^\infty(X_n),r_{X_m}^{X_n}\big)_{m\geq n}:=\prod_{n\in\N}C_P^\infty(X_n)/\mbox{im}\Psi,$$
	where
	$$\Psi:\prod_{n\in\N}C_P^\infty(X_n)\rightarrow\prod_{n\in\N}C_P^\infty(X_n),\Psi((f_n)_{n\in\N}):=(f_n-r_{X_{n+1}}^{X_n}(f_{n+1}))_{n\in\N}.$$
	Hence, it follows with \cite[Section 3.4.4]{Wengenroth} that $P(\partial)$ is surjective on $C^\infty(X)$, i.e.\ $X$ is $P$-convex for supports.
\end{proof}

Before we consider weighted composition operators on $C_P^\infty(X)$ for differential operators $P(\partial)$ for which the set of real zeros of the principal part $P_m$ of $P$ is contained in a one dimensional subspace of $\R^d$ we need to provide an auxiliary result.

\begin{proposition}\label{preparation dense range}
	Let $X\subseteq\R^d$ be open, and let
	$$X_n:=\{x\in X;\,|x|<n\mbox{ and }\dist(x,\R^d\backslash X)>\frac{1}{n}\}, n\in\N.$$
	\begin{enumerate}
		\item[i)] Assume that $d\geq 2$. Then if $X\backslash X_n$ is the disjoint union of a relatively closed subset $F$ of $X$ and a compact set $K$ it follows that $K=\emptyset$.
		\item[ii)] Assume that $d\geq 3$ and that $x_0\in X\backslash X_n$ and $\varepsilon>0$ are such that $\overline{B(x_0,\varepsilon)}\subseteq X\backslash\overline{X}_n$, where $B(x_0,\varepsilon)$ denotes the open ball of radius $\varepsilon$ around $x_0$. Moreover, let $H$ be a hyperplane in $\R^d$ and assume that $C$ is a non-empty, compact connected component of $\big(\R^d\backslash(X_n\cup B(x_0,\varepsilon))\big)\cap H$ which is entirely contained in $X$.
		
		\noindent Then, the connected component $\hat{C}$ of $(\R^d\backslash X_n)\cap H$ which intersects $C$ is also compact and entirely contained in $X$ and for every $x\in\partial_H \hat{C}$, the boundary of $\hat{C}$ relative to $H$, and every $y\in \hat{C}\backslash \partial_H \hat{C}$ we have
		$$\dist(x,\R^d\backslash X)>\dist(y,\R^d\backslash X).$$
	\end{enumerate}
\end{proposition}

\begin{proof}
	i) The claim is clearly true for $X=\R^d$, so we assume without loss of generality that $X\neq \R^d$. Let $X\backslash X_n=F\dot{\cup}K$, where $F$ is a relatively closed subset of $X$ and $K$ is compact. Assume that $x\in K$. Then $\partial K\subseteq \overline{X}_n$ so there are $s<0<t$ such that $x+s e_1\in\overline{X}_n$ and $x+t e_1\in\overline{X}_n$, where $e_1=(\delta_{1,j})_{1\leq j\leq d}$ denotes the first standard basis vector in $\R^d$. Since $\overline{X}_n\subseteq B[0,n]$, the closed ball around the origin of radius $n$, and since $B[0,n]$ is convex, we conclude $x\in B[0,n]$, i.e.\ $|x|\leq n$. 
	
	For $y\in \R^d\backslash X$ we have $y\notin K$ so that there is $t\in[0,1]$ such that
	$$z:=tx+(1-t)y\in\partial K\subseteq\overline{X}_n.$$
	Denoting the euclidean scalar product in $\R^d$ by $\langle\cdot,\cdot\rangle$ we conclude from $y\in \R^d\backslash X$ and $z\in\overline{X}_n$ that $|z-y|^2\geq\frac{1}{n^2}$ so that
	\begin{align*}
		|x-y|^2&=|x-z|^2+2\langle x-z,z-y\rangle+|z-y|^2>2\langle x-z,z-y\rangle+\frac{1}{n^2}\\
		&=2\langle (1-t)(x-y),t(x-y)\rangle+\frac{1}{n^2}\\
		&=2(1-t)t|x-y|^2+\frac{1}{n^2}\geq\frac{1}{n^2}.
	\end{align*}
	Since $y\in \R^d\backslash X$ was chosen arbitrarily, we conclude $\dist (x,\R^d\backslash X)>\frac{1}{n}$ so that
	$$\forall\,x\in K:\,|x|\leq n\mbox{ and } \dist(x,\R^d\backslash X)>\frac{1}{n}.$$
	If for some $x\in K$ we had $|x|=n$ it follows for arbitrary $t>0$ that
	$$x+t\frac{1}{|x|}x\in F$$
	so that by letting $t$ tend to $0$ we obtain $x\in F$ since $F$ is relatively closed in $X$, contradicting that $K$ and $F$ are disjoint. Thus we must have
	$$\forall\,x\in K:\,|x|<n\mbox{ and } \dist(x,\R^d\backslash X)>\frac{1}{n}$$
	which implies $K\subseteq X_n$ yielding the desired contradiction.
	
	ii) Since in particular $d\geq 2$, the compactness of $C\neq \emptyset$ implies that $X\neq\R^d$. We fix $\xi\in C\cap(\R^d\backslash(\overline{X}_n\cup B(x_0,\varepsilon))$. Then $\hat{C}$ is the connected component of $(\R^d\backslash X_n)\cap H$ containing $\xi$. $B(x_0,\varepsilon)\cap H$ is a - possibly empty - open ball in $(\R^d\backslash\overline{X}_n)\cap H$ which we denote by $B^{d-1}(c,r)$ with $c\in H, r\geq 0$, where $r=0$ corresponds to the case when $B(x_0,\varepsilon)\cap H$ is empty.
	
	\textbf{Auxiliary Claim}. We have $C=\hat{C}\cap(\R^d\backslash B(x_0,\varepsilon))\big(=\hat{C}\cap (H\backslash B^{d-1}(c,r))\big)$.
	
	Indeed, by definition it holds $C\subseteq\hat{C}\cap(\R^d\backslash B(x_0,\varepsilon))$. On the other hand, in case $\hat{C}\cap B(x_0,\varepsilon)=\emptyset$, $\hat{C}$ is a connected subset of
	$$(\R^d\backslash X_n)\cap H\cap (\R^d\backslash B(x_0,\varepsilon))=(\R^d\backslash(X_n\cup B(x_0,\varepsilon)))\cap H$$
	which contains $\xi$ and thus
	$$C\supseteq\hat{C}=\hat{C}\cap (\R^d\backslash B(x_0,\varepsilon)).$$
	If $\emptyset\neq\hat{C}\cap B(x_0,\varepsilon)=\hat{C}\cap B^{d-1}(c,r)$ then $\hat{C}\cup B^{d-1}(c,r)$ is connected and thus
	\begin{equation}\label{inclusion d-1 dimensional ball}
		B^{d-1}(c,r)\subseteq\hat{C}\cap \big((\R^d\backslash\overline{X}_n)\cap H\big).
	\end{equation}
	
	We denote by $\tilde{C}$ the connected component of $(\R^d\backslash\overline{X}_n)\cap H$ which contains $\xi$. From inclusion (\ref{inclusion d-1 dimensional ball}) and $\xi\notin B(x_0,\varepsilon)$ it follows that $B^{d-1}(c,r)$ is a proper subset of $\tilde{C}$. Since $(\R^d\backslash\overline{X}_n)\cap H$ is an open subset of the pathwise connnected set $H$ it follows that $\tilde{C}$ is locally pathwise connected and connected, hence a pathwise connected subset of $H$. Since $d-1\geq 2$ and since $B^{d-1}(c,r)$ is an open ball in $H$ and a proper subset of $\tilde{C}$ it follows that $\tilde{C}\backslash B^{d-1}(c,r)$ is a pathwise connected subset of $H$. In particular, $\tilde{C}\backslash B^{d-1}(c,r)$ is a connected subset of $H$. Thus, the same holds for all subsets $M$ of $H$ satisfying
	$$\tilde{C}\backslash B^{d-1}(c,r)\subseteq M\subseteq \overline{\tilde{C}\backslash B^{d-1}(c,r)}^H,$$
	where the latter denotes the closure of $\tilde{C}\backslash B^{d-1}(c,r)$ in $H$. In particular,
	$$\hat{C}\backslash B^{d-1}(c,r)=\hat{C}\cap(\R^d\backslash B(x_0,\varepsilon))$$
	is a connected subset of $H$. By definition of $\hat{C}$, $\hat{C}\cap(\R^d\backslash B(x_0,\varepsilon))$ is a subset of $(\R^d\backslash (X_n\cup B(x_0,\varepsilon)))\cap H$ which contains $\xi$. Because $C$ is the connected component of $(\R^d\backslash (X_n\cup B(x_0,\varepsilon)))\cap H$ which contains $\xi$ we conclude
	$$\hat{C}\cap(\R^d\backslash B(x_0,\varepsilon))\subseteq C$$
	which proves the auxiliary claim.
	
	Now, the auxiliary claim implies
	$$\hat{C}=\big(\hat{C}\cap(\R^d\backslash B(x_0,\varepsilon))\big)\dot{\cup}\big(\hat{C}\cap B(x_0,\varepsilon)\big)\subseteq C\cup\big(\hat{C}\cap\overline{B(x_0,\varepsilon)}\big)\subseteq\hat{C},$$
	thus
	$$\hat{C}=C\cup \big(\hat{C}\cap\overline{B(x_0,\varepsilon)}\big).$$
	$\hat{C}$ is a closed subset of $(\R^d\backslash X_n)\cap H$, thus closed in $\R^d$. Since both $C$ and $\hat{C}\cap\overline{B(x_0,\varepsilon)}$ are compact subsets of $X\cap H$ the same holds for $\hat{C}$. Clearly,
	$$\partial_H\hat{C}\subseteq\partial X_n\cap H$$
	so that
	$$\forall\,x\in \partial_H\hat{C}:\,\dist(x,\R^d\backslash X)\geq \frac{1}{n}.$$
	On the other hand, for $y\in\hat{C}\backslash\partial_H\hat{C}$, the compactness of $\hat{C}$ implies that every half-line in $H$ starting in $y$ intersects $X_n$. By this $y$ can be written as a convex combination of $v,w\in X_n$ which implies $|y|<n$. On the other hand $y\notin X_n$ so that $\dist(y,\R^d\backslash X)\leq 1/n$. If we had $\dist(y,\R^d\backslash X)= 1/n$ it would follow $y\in\partial X_n\cap\hat{C}$ implying $y\in\partial_H\hat{C}$, a contradiction to the choice of $y$. Thus $\dist(y,\R^d\backslash X)< 1/n$ which proves the proposition.
\end{proof}

Part ii) of the next theorem is in particular applicable to non-degenerate parabolic differential operators like the heat operator $\Delta_x-\frac{\partial}{\partial t}$ and to the time dependent free Schr\"odinger operator $\Delta_x+i\frac{\partial}{\partial t}$. 

\begin{theorem}\label{dense range in special cases}
	Let $X\subseteq\R^d$ be open and let $(X_n)_{n\in\N}$ be the open, relatively compact exhaustion of $X$ from Proposition \ref{preparation dense range}. Assume that $x\in X\backslash X_n, \varepsilon>0$ are such that $\overline{B(x,\varepsilon)}\subseteq X\backslash \overline{X}_n$. Then
	$$r_X^{X_n\cup B(x,\varepsilon)}:C_P^\infty(X)\rightarrow C_P^\infty(X_n\cup B(x,\varepsilon))$$
	has dense range in either of the following cases.
	\begin{enumerate}
		\item[i)] $d\geq 2$ and $P$ is elliptic.
		\item[ii)] $X$ is $P$-convex for supports, $d\geq 3$, and the principal part $P_m$ of $P$ satisfies that $\{\xi\in\R^d;\, P_m(\xi)=0\}$ is a one-dimensional subspace of $\R^d$.
	\end{enumerate}
\end{theorem}

\begin{proof}
	For the proof of the claim in case of hypothesis i), assume that $X\backslash(X_n\cup B(x,\varepsilon))$ can be written as the disjoint union of a relatively closed subset $F$ of $X$ and a compact subset $K$ of $X$, i.e.\ assume that
	\begin{equation}\label{Lax-Malgrange decomposition}
	X\backslash(X_n\cup B(x,\varepsilon))=F\dot{\cup}K.
	\end{equation}
	We will show that necessarily $K=\emptyset$.
	
	So if (\ref{Lax-Malgrange decomposition}) holds then $\partial B(x,\varepsilon)\subset F\dot{\cup}K$ and
	\begin{align*}
	X\backslash X_n&=\big((X\backslash X_n)\cap(X\backslash B(x,\varepsilon)\big)\dot{\cup}\big((X\backslash X_n)\cap B(x,\varepsilon)\big)\\
	&=X\backslash (X_n\cup B(x,\varepsilon))\dot{\cup} B(x,\varepsilon)=F\dot{\cup}K\dot{\cup}B(x,\varepsilon).
	\end{align*}
	Since $\partial B(x,\varepsilon)\subseteq F\dot{\cup}K$ and since $\partial B(x,\varepsilon)$ is connected (recall that we assume $d\geq 2$) it follows that $\partial B(x,\varepsilon)$ is contained in a single connected component of $F\dot{\cup}K$. In particular we either have $\partial B(x,\varepsilon)\subseteq F$ or $\partial B(x,\varepsilon)\subseteq K$ so either $F\cup B(x,\varepsilon)=F\cup B[x,\varepsilon]$ is relatively closed in $X$ or $K\cup B(x,\varepsilon)=K\cup B[x,\varepsilon]$ is compact.
	
	In the first case we have
	$$X\backslash X_n=F\dot{\cup}B(x,\varepsilon)\dot{\cup}K=\big(F\cup B[x,\varepsilon]\big)\dot{\cup}K$$
	in the second case we conclude
	$$X\backslash X_n=F\dot{\cup}\big(B[x,\varepsilon]\cup K\big).$$
	Since $B[x,\varepsilon]\neq\emptyset$, the latter case cannot occur by Proposition \ref{preparation dense range} i). In the first case it follows again from Proposition \ref{preparation dense range} i) that $K=\emptyset$.
	
	Thus, for every decomposition of $X\backslash(X_n\cup B(x,\varepsilon))$ as in (\ref{Lax-Malgrange decomposition}) it follows that $K=\emptyset$. Therefore, by the Lax-Malgrange Theorem (see e.g.\ \cite[Theorem 4.4.5 combined with the remark preceding Corollary 4.4.4 resp.\ with Theorem 8.6.1]{Hoermander} or \cite[Theorem 3.10.7]{Narasimhan}) it follows that the restriction map
	$$r_X^{X_n\cup B(x,\varepsilon)}:C_P^\infty(X)\rightarrow C_P^\infty(X_n\cup B(x,\varepsilon))$$
	has dense range.
	
	Next, we prove the claim for case ii). Since $X$ is supposed to be $P$-convex for supports, as in the proof of Theorem \ref{good exhaustion}, it follows that $X_n, n\in\N,$ is $P$-convex for supports as well, i.e.\ $P(\partial)$ is surjective on $C^\infty(X_n)$. Since $B(x,\varepsilon)$ is convex, $P(\partial)$ is surjective on $C^\infty(B(x,\varepsilon))$ and because $B(x,\varepsilon)$ and $X_n$ are disjoint, $P(\partial)$ is surjective on $X_n\cup B(x,\varepsilon)$, thus $X_n\cup B(x,\varepsilon)$ is $P$-convex for supports.
	
	Assume that for some characteristic hyperplane
	$$H=\{x\in\R^d;\,\langle N,x\rangle=\alpha\}$$
	for $P$, i.e.\ $N\in\R^d\backslash\{0\}, \alpha\in\R$, and $P_m(N)=0$, there is a non-empty, compact connected component of $(\R^d\backslash (X_n\cup B(x,\varepsilon)))\cap H$ which is entirely contained in $X$. Then, by Proposition \ref{preparation dense range} ii) there is a compact subset $\hat{C}$ of $X\cap H$ such that
	$$\forall\,v\in\partial_H\hat{C}, y\in\hat{C}\backslash\partial_H\hat{C}:\dist(v,\R^d\backslash X)>\dist(y,\R^d\backslash X),$$
	where $\partial_H\hat{C}$ denotes the boundary of $\hat{C}$ in $H$. By \cite[Theorem 10.8.1]{Hoermander} this contradicts the $P$-convexity for supports of $X$.
	
	Thus, for every characteristic hyperplane $H$ for $P$ there is no non-empty, compact connected component of $(\R^d\backslash (X_n\cup B(x,\varepsilon)))\cap H$ which is entirely contained in $X$. From \cite[Theorem 1]{Kalmes18} it follows that $r_X^{X_n\cup B(x,\varepsilon)}$ has dense range.
\end{proof}

We are now ready to characterize ($m$-)topologizability and power boundedness of weighted composition operators on kernels of certain partial differential operators. We begin with kernels of elliptic operators.

\begin{theorem}\label{power bounded elliptic}
	Let $P\in\C[X_1,\ldots,X_d], d\geq 2,$ be elliptic with $P(0)=0$, $X\subseteq\R^d$ be open, $w\in C(X)$ and $\psi:X\rightarrow X$ be continuous such that $w$ does not vanish identically on any connected component of $X$, $\psi$ is locally injective, and $C_{w,\psi}$ is well-defined on $C_P^\infty(X)$.
	\begin{enumerate}
		\item[a)] The following are equivalent.
			\begin{enumerate}
				\item[i)] $C_{w,\psi}$ is $m$-topologizable on $C_P^\infty(X)$.
				\item[ii)] $C_{w,\psi}$ is topologizable on $C_P^\infty(X)$.
				\item[iii)] $\psi$ has stable orbits.
			\end{enumerate}
		\item[b)] The following are equivalent.
		\begin{enumerate}
			\item[i)] $C_{w,\psi}$ is power bounded on $C_P^\infty(X)$.
			\item[ii)] $C_{w,\psi}$ is topologizable and $\{C_{w,\psi}^m(w);\,m\in\N\}$ is bounded in $C(X)$.
		\end{enumerate}
	\end{enumerate}
\end{theorem}

\begin{proof}
	Since $P(0)=0$ it follows that constant functions belong to $C_P^\infty(X)$ so that from the hypothesis we conclude $w\in C_P^\infty(X)$. Since the kernel of an elliptic differential operator consists of real analytic functions (see \cite[Theorem 8.6.1]{Hoermander}) and since $w$ does not vanish identically on any connected component of $X$ it follows that $w^{-1}(\K\backslash\{0\})$ is dense in $X$. Moreover, by hypothesis on $\psi$ and Brouwer's Invariance of Domain Theorem it follows that condition ii) from Proposition \ref{denseness result} is satisfied so that Proposition \ref{denseness result} is applicable to conclude that condition c) from Corollary \ref{characterizing topologizability for compact-open} holds. As observed above, condition b) from Corollary \ref{characterizing topologizability for compact-open} is also valid. So having in mind that the topology of $C_P^\infty(X)$ is the compact-open topology we only have to show that condition a) of Corollary \ref{characterizing topologizability for compact-open} is fulfilled in order to prove part a) of the theorem. But this condition is satisfied due to Theorem \ref{dense range in special cases} i) so that a) follows. Part b) of the theorem follows from Theorem \ref{power boundedness for compact-open}.
\end{proof}

Our next aim is to characterize $m$-topologizability as well as power boundedness for weighted composition operators defined on kernels of certain non-elliptic operators. Part a) of the next theorem is in particular applicable to non-degenerate  parabolic operators like the heat operator while part b) covers solutions to the time dependent free Schr\"odinger equation. For a characterization of $P$-convexity for supports of an open subset $X\subseteq\R^d$ for those differential operators discussed in the next theorem, see \cite{Kalmes16}.

\begin{theorem}\label{one-dimensional characteristics}
	Let $d\geq 3, P\in\C[X_1,\ldots,X_d]\backslash\{0\}$ with principal part $P_m$ such that $\{\xi\in\R^d;\,P_m(\xi)=0\}$ is a one-dimensional subspace of $\R^d$ and assume that $P(0)=0$. Moreover, let $X\subseteq\R^d$ be open and $P$-convex for supports, and let $w\in C(X)$ and $\psi:X\rightarrow X$ be continuous such that $C_{w,\psi}$ is defined on $C_P^\infty(X)$.
	\begin{enumerate}
		\item[a)] Let $P$ be hypoelliptic and assume that
		$$\forall\,m\in\N_0:\,\{x\in X;\,w(\psi^m(x))\neq 0\}$$
		is dense in $X$. Then, the following are equivalent.
		\begin{enumerate}
			\item[i)] $C_{w,\psi}$ is $m$-topologizable on $C_P^\infty(X)$.
			\item[ii)] $C_{w,\psi}$ is topologizable on $C_P^\infty(X)$.
			\item[iii)] $\psi$ has stable orbits.
		\end{enumerate}
		Moreover, the following are equivalent, too.
		\begin{enumerate}
			\item[iv)] $C_{w,\psi}$ is power bounded on $C_P^\infty(X)$.
			\item[v)] $C_{w,\psi}$ is topologizable and $\{C_{w,\psi}^m(w);\,m\in\N\}$ is bounded in $C(X)$.
		\end{enumerate}
		
		\item[b)] Let $P$ be not hypoelliptic and assume that
		$$\forall\,l\in\N_0:\,\{x\in X;\,w(\psi^l(x))\neq 0\}$$
		is dense in $X$ and that $\psi_c\in C_P^\infty(X)$ for all $1\leq c\leq d$. Then, the following are equivalent.
		\begin{enumerate}
			\item[i)] $C_{w,\psi}$ is power bounded on $C_P^\infty(X)$.
			\item[ii)] $C_{w,\psi}$ is topologizable and $\{C_{w,\psi}^m(w);\,m\in\N\}$ as well as $\{C_{w,\psi}^m(\psi_c);\,m\in\N\}, 1\leq c\leq d,$ are bounded in $C^\infty(X)$.
			\item[iii)] $\psi$ has stable orbits and $\{C_{w,\psi}^m(w);\,m\in\N\}$, $\{C_{w,\psi}^m(\psi_c);\,m\in\N\}, 1\leq c\leq d,$ are bounded in $C^\infty(X)$.
		\end{enumerate}
	\end{enumerate}
\end{theorem}

\begin{proof}
	Since $P(0)=0$ and since $C_{w,\psi}$ is defined on $C_P^\infty(X)$ it follows that $w\in C_P^\infty(X)$. Now, replacing the reference to Theorem \ref{dense range in special cases} i) by Theorem \ref{dense range in special cases} ii) (and in the proof of part b) the reference to Theorem \ref{power boundedness for compact-open} by Theorem \ref{power boundedness for smooth}) the theorem follows along the same arguments as does Theorem \ref{power bounded elliptic}.
\end{proof}

For the special case of the Cauchy-Riemann operator the next result is due to Bonet and Doma\'nski, see \cite{BonetDomanski11}.

\begin{theorem}\label{mean ergodic on hypoelliptic kernels}
	Let $P\in\C[X_1,\ldots,X_d]\backslash\{0\}, d\geq 2,$ be a hypoelliptic polynomial such that $\{\xi\in\R^d;\,P_m(\xi)=0\}$ is contained in a one-dimensional subspace of $\R^d$. Moreover, let $X\subseteq\R^d$ be $P$-convex for supports and let $\psi:X\rightarrow X$ be smooth such that the composition operator $C_\psi$ operates on $C_P^\infty(X)$. Moreover, let $d\geq 3$ if $P$ is not elliptic. Then the following are equivalent.
	\begin{enumerate}
		\item[i)] $C_\psi$ is power bounded.
		\item[ii)] $C_\psi$ is uniformly mean ergodic.
		\item[iii)] $C_\psi$ is mean ergodic.
		\item[iv)] $C_\psi$ is $m$-topologizable.
		\item[v)] $C_\psi$ is topologizable.
		\item[vi)] $\psi$ has stable orbits.
	\end{enumerate}
\end{theorem}

\begin{proof}
	As a closed subspace of the nuclear Fr\'echet space $C^\infty(X)$, the space $C_P^\infty(X)$ is a nuclear Fr\'echet space, in particular a Fr\'echet-Montel space. Since $P$ is hypoelliptic, the topology of $C_P(X)$ is the compact-open topology. By Theorem \ref{dense range in special cases}, hypothesis a) of Corollary \ref{mean ergodic for compact-open} is satisfied as it is trivially hypotheses b). Thus the claim follows from Corollary \ref{mean ergodic for compact-open}. 
\end{proof}

We close this section by a characterization of those weighted composition operators $C_{w,\psi}$ which are defined on the kernel of the heat operator in terms of the weight function $w$ and the symbol $\psi$. For an analogous result for the Cauchy-Riemann operator as well as for the Laplace operator, see \cite[Proposition 6.6]{Kalmes17}. In order to keep the usual notation employed in the context of the heat operator, we consider $\R^{d+1}$ and write the elements of $\R^{d+1}$ as $(t,x)$ or $(x_0,x)$ with $t,x_0\in\R$ and $x=(x_1,\ldots,x_d)\in\R^d$. Furthermore, we write $\nabla_x f$ and $\Delta_x f$ for $(\partial_1 f,\ldots,\partial_d f)$ and $\sum_{j=1}^d\partial_j^2 f$, respectively.

\begin{proposition}\label{invariance for the heat operator}
	Let $X\subseteq\R^{d+1}$ be open and let $w\in C^2(X)$ and $\psi:X\rightarrow X$ be $C^2$. Then, the following are equivalent.
	\begin{enumerate}
		\item[i)] $C_{w,\psi}$ is defined on $C_P^\infty(X)$ for $P(\xi)=\xi_0-\sum_{j=1}^d\xi_j^2$, i.e.\ on the kernel of the heat operator $\frac{\partial}{\partial t}-\Delta_x$.
		\item[ii)] The following conditions are satisfied.
		\begin{enumerate}
			\item[a)] $w\in C_P^\infty(X)$ as well as $w\psi_j\in C_P^\infty(X)$ for all $1\leq j\leq d$.
			\item[b)] $P(\partial)(w\psi_0)=w|\nabla_x\psi_1|^2$.
			\item[c)] $w|\nabla_x\psi_0|^2=0$ and $w|\nabla_x\psi_j|^2=w|\nabla_x\psi_k|^2$ for all $1\leq j,k\leq d$.
			\item[d)] $w\langle\nabla_x \psi_j,\nabla_x\psi_k\rangle=0$ for all $0\leq j\neq k\leq d$.
		\end{enumerate}
		
	\end{enumerate}
\end{proposition}

\begin{proof}
	In order to simplify notation we denote the heat operator simply by $H$ and write $f(\psi)$ instead of $f\circ\psi$. A straightforward calculation shows that for every $u\in C^2(X)$
	\begin{align}\label{heat of composition}
		H(C_{w,\psi}(u))&=H(w)\cdot u(\psi)+w\sum_{l=0}^d\partial_l u(\psi) H(\psi_l)\nonumber\\
		&-2\sum_{l=0}^d \partial_l u(\psi)\langle\nabla_x w,\nabla_x \psi_l\rangle-w\sum_{l,m=0}^d\partial_l\partial_m u(\psi)\langle\nabla_x \psi_l,\nabla_x \psi_m\rangle
	\end{align}
	Assume that i) holds. Since $u=1\in C_P^\infty(X)$ we obtain $w\in C_P^\infty(X)$. Using this and that $u(x)=x_c\in C_P^\infty(X), 1\leq c\leq d,$ we derive $0=H(w\psi_c)=wH(\psi_c)-2\langle\nabla_x w,\nabla_x\psi_c\rangle$ for all $1\leq c\leq d$.
	
	For $\zeta\in\C^{d+1}$ we denote $e_\zeta(x):=\exp(\sum_{j=0}^d\zeta_j x_j)$. For $j,k\in\{1,\ldots,d\}, j\neq k$ let $\zeta_j=1,\zeta_k=\pm i$ and $\zeta_l=0$ for $l\neq j,k$. Then $e_\zeta\in C_P^\infty(X)$ and (\ref{heat of composition}) gives
	\begin{align*}
		0&=w\big(H(\psi_j)\pm i H(\psi_k)\big)-2\big(\langle\nabla_x w,\nabla_x\psi_j\rangle\pm i\langle\nabla_x w,\nabla_x\psi_k\rangle\big)\\
		&\quad -w\big(|\nabla_x \psi_j|^2-|\nabla_x\psi_k|^2\pm 2i\langle\nabla_x\psi_j,\nabla_x\psi_k\rangle\big).
	\end{align*}
	Substracting the "-" version of the previous equality from the "+" version and using $0=w H(\psi_k)-2\langle\nabla_x w,\nabla_x\psi_k\rangle$ we obtain
	\begin{align*}
		0&= w2iH(\psi_k)-4i\langle \nabla_x w,\nabla_x\psi_k\rangle -4i w\langle\nabla_x\psi_j,\nabla_x\psi_k\rangle\\
		&=-4i w\langle\nabla_x\psi_j,\nabla_x\psi_k\rangle,
	\end{align*}
	thus
	\begin{equation}\label{equation 1}
		\forall\,1\leq j\neq k\leq d:\,w\langle\nabla_x\psi_j,\nabla_x\psi_k\rangle=0.
	\end{equation}
	Next, for fixed $j\in\{1,\ldots,d\}$ and $\zeta_0=-1, \zeta_j=\pm i$, and $\zeta_k=0$ for $k\neq 0,j$ we have $e_\zeta\in C_P^\infty(X)$ and evaluating (\ref{heat of composition}) yields
	\begin{align}\label{equation 2}
		0&= -w H(\psi_0)\pm i w H(\psi_j)+2\langle\nabla_x w,\nabla_x\psi_0\rangle\mp 2i\langle \nabla_x w,\nabla_x\psi_j\rangle\nonumber\\
		&\quad -w\big(|\nabla_x \psi_0|^2-|\nabla_x\psi_j|^2\mp 2i\langle\nabla_x\psi_0,\nabla_x\psi_j\rangle\big)\\
		&=-w H(\psi_0)+2\langle\nabla_x w,\nabla_x\psi_0\rangle-w\big(|\nabla_x \psi_0|^2-|\nabla_x\psi_j|^2\mp 2i\langle\nabla_x\psi_0,\nabla_x\psi_j\rangle\big),\nonumber
	\end{align}
	where we have used $0=H(w\psi_j)=wH(\psi_j)-2\langle\nabla_x w,\nabla_x\psi_j\rangle$. Substracting again the "-" version from the "+" version of this equality we obtain $0=-w 4i\langle\nabla_x\psi_0,\nabla_x\psi_j\rangle$, i.e.\
	\begin{equation}\label{equation 3}
		\forall\,1\leq j\leq d:\,w \langle\nabla_x\psi_0,\nabla_x\psi_j\rangle=0.
	\end{equation}
	Using $H(w)=0$, $0=H(w\psi_c)=wH(\psi_c)-2\langle\nabla_x w,\nabla_x\psi_c\rangle, c=1,\ldots,d$, (\ref{equation 1}), and (\ref{equation 3}) it follows that (\ref{heat of composition}) reduces to
	\begin{align}\label{heat of composition simplified}
		H(C_{w,\psi}(u))&=w\sum_{l=0}^d\partial_l u(\psi) H(\psi_l)-2\sum_{l=0}^d\partial_l u(\psi)\langle\nabla_x w,\nabla_x\psi_l\rangle\nonumber\\
		&\quad -w\sum_{l=0}^d\partial_l^2 u(\psi)|\nabla_x\psi_l|^2\\
		&=\partial_0 u(\psi)\big(w H(\psi_0)-2\langle\nabla_x w,\nabla_x\psi_0\rangle\big)-w\sum_{l=0}^d\partial_l^2 u(\psi)|\nabla_x\psi_l|^2\nonumber.
	\end{align}
	Next, we set $\zeta_j=1, \zeta_k=i$ for $1\leq j\neq k\leq d$ as well as $\zeta_l=0$ for $l\neq j,k$. Then $e_\zeta\in C_P^\infty(X)$ and evaluating (\ref{heat of composition simplified}) yields
	\begin{equation}\label{equation 4}
		0=-w(|\nabla_x\psi_j|^2-|\nabla_x\psi_k|^2)
	\end{equation}
	because $\delta_0 e_\zeta=0$, so that (\ref{heat of composition simplified}) further simplifies to
	\begin{align}\label{heat of composition simplified 2}
		H(C_{w,\psi}(u))&=\partial_0 u(\psi)\big(w H(\psi_0)-2\langle\nabla_x w,\nabla_x\psi_0\rangle\big)\nonumber\\
		&\quad -w\big(\partial_0^2 u(\psi)|\nabla_x\psi_0|^2+|\nabla_x\psi_1|^2\Delta_x u(\psi)\big).
	\end{align}
	Finally, we choose $\zeta_0=i, \zeta_1=\sqrt{i},$ for any square root of $i$, and $\zeta_k=0$ for $k\neq 0,1$. Then $e_\zeta\in C_P^\infty(X)$ so that from (\ref{heat of composition simplified 2}) we derive
	\begin{align*}
		0&=i\big(w H(\psi_0)-2\langle\nabla_x w,\nabla_x\psi_0\rangle\big)-w\big(-|\nabla_x\psi_0|^2+i|\nabla_x\psi_1|^2\big).
	\end{align*}
	Because $w,\psi_0$, and $\psi_1$ are real valued it follows that
	\begin{equation}\label{psi_0 only depends on x_0}
		0=w|\nabla_x\psi_0|^2=w\sum_{j=1}^d (\partial_j\psi_0)^2,
	\end{equation}
	i.e.\ on the set $\{x\in X;\,w(x)\neq 0\}$ $\psi_0$ only depends on $x_0$. With this, (\ref{heat of composition simplified 2}) simplifies even further to
	\begin{equation}\label{heat of composition simplified 3}
		H(C_{w,\psi}(u))=\partial_0 u(\psi) \big(w H(\psi_0)-2\langle\nabla_x w,\nabla_x\psi_0\rangle\big)-\Delta_x u(\psi) w|\nabla_x\psi_1|^2.
	\end{equation}
	Inserting any $u\in C_P^\infty(X)$ for which $\partial_0 u$ does not have any zero in $X$ into the previous equality, since $\partial_0 u=\Delta_x u$, we obtain
	\begin{equation}\label{last piece}
		H(w\psi_0)-w|\nabla_x\psi_1|^2=0.
	\end{equation}
	Thus, i) implies ii) by $H(w)=0$, $H(w\psi_j)=0, 1\leq j\leq d$, (\ref{last piece}), (\ref{psi_0 only depends on x_0}), (\ref{equation 4}), (\ref{equation 1}), and (\ref{equation 3}).
	
	On the other hand, if ii) a), c), and d) hold, it follows that (\ref{heat of composition}) simplifies to (\ref{heat of composition simplified 3}) which by b) implies i).
\end{proof}

\section{Weighted composition operators on $C(X)$ as generators of strongly continuous semigroups}\label{generators}

In this final section we combine recent results by Goli\'nska and Wegner \cite{GolinskaWegner} as well as by Frerick et al.\ \cite{FrerickJordaKalmesWengenroth} with the results from section \ref{power boundedness} to characterize under mild additional assumptions on the weight and the symbol those weighted composition operators on $C(X)$ which are generators of strongly continuous operator semigroups on $C(X)$.

The definition and basic results for semigroups of operators on locally convex spaces $E$ are the same as for Banach spaces, see e.g.\ \cite{Yosida}. However, contrary to the case of Banach spaces, on a locally convex space $E$ not every continuous linear operator $A$ generates a strongly continuous semigroup $T=(T_t)_{t\geq 0}$ and even if it does, the semigroup $T$ need not be of the form $T_t(x)=\exp(t A)x=\sum_{k=0}^\infty\frac{t^k}{k!}A^k(x)$ (see e.g.\ \cite[Example 1, Example 4]{FrerickJordaKalmesWengenroth}).

\begin{theorem}\label{generator}
	Let $X$ be a locally compact, $\sigma$-compact, non-compact Hausdorff space, $w\in C(X)$, $\psi:X\rightarrow X$ continuous such that $w^{-1}(\K\backslash\{0\})$ is dense in $X$ and such that for every $x\in X$ there is an open neighborhood $U_x$ of $x$ in $X$ such that $\psi_{|U_x}$ is injective and open. Then the following are equivalent when $C(X)$ is equipped with the compact-open topology.
	\begin{enumerate}
		\item[i)] $C_{w,\psi}$ generates a uniformly continuous semigroup on $C(X)$.
		\item[ii)] $C_{w,\psi}$ generates a strongly continuous semigroup on $C(X)$.
		\item[iii)] $C_{w,\psi}$ is $m$-topologizable on $C(X)$.
		\item[iv)] $C_{w,\psi}$ is topologizable on $C(X)$.
		\item[v)] $\psi$ has stable orbits.
	\end{enumerate}
	Moreover, if $C_{w,\psi}$ generates a strongly continuous semigroup $T=(T_t)_{t\geq 0}$ on $C(X)$, then
	$$\forall\,t\geq 0, f\in C(X): T_t(f)=\exp(t A)f=\sum_{k=0}^\infty \frac{t^k}{k!}C_{w,\psi}^k(f),$$
	where the series converges in the compact-open topology.
\end{theorem}

\begin{proof}
	Let $(K_n)_{n\in\N}$ be a compact exhaustion of $X$ such that the interior of $K_n$ is not empty for each $n\in\N$. For $m,n\in\N, m\geq n,$ let $\pi_m^n:C(K_m)\rightarrow C(K_n), f\mapsto f_{|K_n}$. Equipping $C(K_n)$ etc. with the supremum norm, $\pi_m^n$ is a linear contraction and $C(X)$ is topologically isomorphic to the closed subspace
	$$\{(f_n)_{n\in\N}\in\prod_{n\in\N}C(K_n);\,\pi_m^n(f_m)=f_n\mbox{ for all }n\leq m\}$$
	of the product space $\prod_n C(K_n)$. By Tietze's Extension Theorem, each $\pi_m^n, n\leq m$ is surjective so that $C(X)$ is a quojection.
	
	It follows from \cite[Theorem 2]{FrerickJordaKalmesWengenroth} that i) and ii) are equivalent and that in case $C_{w,\psi}$ generates a strongly continuous semigroup it is given by the exponential series. 
	
	Moreover, if $C_{w,\psi}$ generates a strongly continuous semigroup on $C(X)$ it follows from \cite[Theorem 2]{FrerickJordaKalmesWengenroth} that
	\begin{equation}\label{FJKW-1}
		\forall\,n\in\N\,\exists\,m\in\N\,\forall\,k\in\N, f\in C(X):\,f_{|K_m}=0\Rightarrow C_{w,\psi}^k(f)_{|K_n}=0.
	\end{equation}
	Since
	$$C_{w,\psi}^k(f)=\prod_{j=0}^{k-1}w(\psi^j(\cdot)) f(\psi^k(\cdot))$$
	it follows from the hypothesis on $w$ and $\psi$, together with Proposition \ref{denseness result}, that
	$$C_{w,\psi}^k(f)_{|K_n}=0\Rightarrow \forall\,x\in K_n:\,f(\psi^k(x))=0,$$
	so that (\ref{FJKW-1}) becomes
	\begin{equation}\label{FJKW-2}
		\forall\,n\in\N\,\exists\,m\in\N\,\forall\,k\in\N, f\in C(X):\,f_{|K_m}=0\Rightarrow f(\psi^k(\cdot))_{|K_n}=0.
	\end{equation}
	Now we fix $n\in\N$ and choose $m\in\N$ according to (\ref{FJKW-2}). By Urysohn's Lemma, there is $f\in C(X)$ such that $f_{|K_m}=0$ and $f_{|X\backslash\mbox{int}(K_{m+1})}=1$, where $\mbox{int}(K_{m+1})$ denotes the interior of $K_{m+1}$. Evaluating (\ref{FJKW-2}) for this particular $f$ yields
	$$\forall\, k\in\N:\,\psi^k(K_n)\subseteq K_{m+1}.$$
	Since $n\in\N$ was arbitrary, it follows that ii) implies v).
	
	With the aid of Tietze's Extension Theorem it follows from Corollary \ref{characterizing topologizability for compact-open} that iii), iv), and v) are equivalent, so that it only remains to show that iii) implies ii). However, this implication follows immediately from \cite[Theorem 1]{GolinskaWegner}.
\end{proof}

\subsection*{Acknowledgements}
The author would like to thank E.\ Jord\'a and D.\ Jornet from the Universitat Polit\`ecnica de Val\`encia for stimulating discussions about the topic of this article. Moreover, the author is very grateful to his colleagues from Val\`encia for the cordial hospitality during his numerous stays there. Finally, the author would like to thank the anonymous referee for a careful reading of the article and for his comments and suggestions which helped to improve the presentation of the results.

\end{document}